\documentclass{siamart0516}

\usepackage{graphicx}

\usepackage[revision]{revdiff}

\usepackage{amsfonts,amssymb,amsopn}
\usepackage{algorithmic}
\usepackage{color,xcolor,graphicx}
\usepackage{multirow,multicol}
\usepackage{enumerate}
\usepackage{epstopdf}
\usepackage{adjustbox}
\usepackage{subfigure}

\def\Diff{\mathrm{d}}
\def\Fun{\mathcal{F}}
\def\Mfd{\mathcal{M}}
\def\tr{\mathrm{tr}}

\def\Real{\mathbb{R}}
\def\Expect{\mathbb{E}}
\def\Prob{\mathbb{P}}
\newsiamthm{claim}{Claim}
\newsiamthm{remark}{Remark}
\ifpdf
  \DeclareGraphicsExtensions{.eps,.pdf,.png,.jpg}
\else
  \DeclareGraphicsExtensions{.eps}
\fi

\newcommand{\TheTitle}{Global optimization with orthogonality constraints via stochastic diffusion on manifold}
\newcommand{\ShortTitle}{Global optimization with orthogonality constraints}
\newcommand{\TheAuthors}{H. Yuan, X. Gu, R. Lai, and Z. Wen}

\headers{\ShortTitle}{\TheAuthors}

\title{{\TheTitle}
}

\author{
  Honglin Yuan\thanks{Computational Mathematics, Peking University, China (\email{yhlmath@pku.edu.cn}, \email{xiaoyigu@pku.edu.cn}).}
  \and
  Xiaoyi Gu\footnotemark[1]
  \and
  Rongjie Lai\thanks{Department of Mathematics, Rensselaer Polytechnic Institute, Troy, NY 12180 (\email{lair@rpi.edu}).}
  \and
  Zaiwen Wen\thanks{Beijing International Center for Mathematical Research, Peking University, China (\email{wenzw@pku.edu.cn}).}
}

\ifpdf
\hypersetup{
  pdftitle={\TheTitle},
  pdfauthor={\TheAuthors}
}
\fi

\begin{document}

\maketitle

\begin{abstract}
Orthogonality constrained optimization is widely used in applications from science and engineering. Due to the nonconvex orthogonality constraints, many numerical algorithms often can hardly achieve the global optimality. We aim at establishing an efficient scheme for finding global minimizers under one or more orthogonality constraints. The main concept is based on noisy gradient flow constructed from stochastic differential equations (SDE) on the Stiefel manifold, the differential geometric characterization of orthogonality constraints. We derive an explicit representation of SDE on the Stiefel manifold endowed with a canonical metric and propose a numerically efficient scheme to simulate this SDE based on Cayley transformation with theoretical convergence guarantee. The convergence to global optimizers is proved under second-order continuity. The effectiveness and efficiency of the proposed algorithms are demonstrated on a variety of problems including homogeneous polynomial optimization, computation of stability number, and 3D structure determination from Common Lines in Cryo-EM.  \end{abstract}


\section{Introduction}
\label{sec:intro}
Mathematically, the orthogonality constrained problem can be formulated as the following form:
\begin{equation}
	\min_{X \in \Real^{n \times p}} \Fun (X),~ \quad \text{s.t.}~ X^\top X = I_p,
	\label{eq:prob:main}
\end{equation}
where  $\mathcal{F}$ is a smooth objective function and $I_p$ indicates the $p$-by-$p$ identity matrix. The feasible set $\Mfd_{n,p}=\{X \in \Real^{n \times p}: X^\top X = I_p\}$ is well-known as the Stiefel manifold (once equipped with its natural submanifold structure from $\Real^{n \times p}$). We also denote it by $\Mfd$ if there is no ambiguity on the dimensions. 

Particularly in the case of $p = 1$, the above problem is known as the
spherically constrained problem. In the case of $p = n$, the feasible set
becomes orthogonal group $\mathcal{O}_n$, where the feasible matrices are square
and orthogonal. More generally, the following optimization problem with multiple
orthogonality (or spherical) constraints is widely used in many problems such as
conformal mapping~\cite{gu2003global,lai2014folding}, p-Harmonic
flow~\cite{lin1989relaxation,tang2001color,vese2002numerical,goldfarb2009curvilinear},
1-bit compressive sensing~\cite{boufounos20081,laska2011trust}, compressed
modes~\cite{ozolicnvs2013compressed}, the graph stability number, Cryo-electron
microscopy (Cryo-EM)~\cite{Singer:2011ba}, nonlinear eigenvalue problem in
density functional theory \cite{scf-sinum,TRDFT} as well as dictionary learning~\cite{aharon2006rm,cai2014data}:
\begin{equation}
	\min_{X_1 \in \Real^{n_1 \times p_1},\dots,X_q \in \Real^{n_q \times p_q}} \Fun(X_1,\dots,X_q),~ \text{s.t.}~ X_i^\top X_i = I_{p_i},~ i = 1,\dots,q.
	\label{eq:prob:multi}
\end{equation}

Non-convexity is one of the major challenges of problems \cref{eq:prob:main} and
\cref{eq:prob:multi} since there might be mulitiple local minimizers, from which
finding global minimizers is generally NP-hard. Most existing
algorithms \cite{Absil:2007uy,Wen:2012ga,NocedalWright06} on the Stiefel manifold focus on finding local optimizers without exploiting the global structures, and thus there is no guarantee to obtain the global minimizers except for some trivial cases.

\subsection{Local feasible solver on Stiefel manifold}
\label{subsec:OptStfMfd}
%
%

One step of our algorithm is mostly based on the first-order algorithms proposed in \cite{Wen:2012ga}, which we consider to have low computational cost and briefly describe here. 

Given a point $X$ on $\Mfd_{n,p}$, the canonical metric $g^c$ on the tangent space $\mathcal{T}_X\Mfd_{n,p} =  \{Z\in \Real^{n,p},Z^\top X+X^\top Z=0\}$ is defined as
\[
	\begin{aligned}
		g^c(Z_1,Z_2) & :=\tr(Z^\top_1(I-\frac{1}{2}XX^\top)Z_2).
	\end{aligned}
\]
This metric considers the Stiefel manifold to be a quotient space $M_{n,p} = O_n/O_{n-p}$, in which $O_k$ is the group of $k\times k$ orthogonal matrices. 
Let us write $G_{ij}=\partial_{ij}\Fun$, then the gradient of $\Fun$ with
respect to $g^c$ \cite{Absil:2007uy,Wen:2012ga}  is given by 
\begin{equation}
	\nabla_{\mathcal{M}}^c\Fun (X)=G-XG^\top X.
	\label{eq:grad:c}
\end{equation}
Throughout the paper, we will always adopt the canonical metric, the superscript indicating the canonical metric will be omitted.
Let $A = GX^\top-XG^\top$, the authors in \cite{Wen:2012ga} consider an implicit update scheme as 
\begin{equation}
	Y(\tau) = X - \tau A(\frac{X+Y(\tau)}{2}).
	\label{eq:gradflow}
\end{equation}
This leads to the following Cayley transformation,
\begin{equation}
	Y(\tau) = (I+\frac{\tau}{2}A)^{-1}(I-\frac{\tau}{2}A)X.
    \label{eq:gradflowout}
\end{equation}
It has been shown in \cite{Wen:2012ga} that the update scheme automatically preserve the orthogonality constraint due to the property of Cayley transformation. In addition, with certain conditions, it has also been proved in \cite{Wen:2012ga} that the sequence generated using this algorithm satisfying $\lim_{k\rightarrow 0 }\|\nabla F(X^k) \|_F = 0$. 

\subsection{Global optimization by diffusions}
\label{subsec:GlobalOptDfu}
For a general non-convex unconstrained optimization problem 
\begin{equation}
	\min_{x\in\Real^n} f(x).
	\label{eq:prob:orig}
\end{equation}
A well known method is to consider the gradient flow
\begin{equation}
	\Diff x(t) = -\nabla f(x(t)) \Diff t,
\end{equation}
yet $x(t)$ is often trapped at a local stationary point due to nonconvexity. 
A well-known remedy is to add white noise to the gradient flow \cite{AluffiPentini:1985eq,Chiang:1987ip,Geman:1986js,Gidas:1985kz}, allowing the trajectory to ``climb over the mountains'' and escape from the local minimizers. Mathematically, this type of methods can be formulated as the following \emph{Stochastic Differential Equation} (SDE) \cite{oksendal:2003jp}:
\begin{equation}
	\Diff x(t) = -\nabla f(x(t)) \Diff t + \sigma(t) \Diff B(t),
	\label{eq:sde:rn}
\end{equation}
where $f(x)$ is the objective function defined on $\Real^n$ and $B(t)$ is an $n$-dimensional standard Brownian motion, which is also known as the Wiener process. Different choices of $\sigma(t)$ lead to different diffusion algorithms and different results. 
It has been proved in \cite{Chiang:1987ip,Geman:1986js} that if the diffusion
strength $\sigma(t)$ is chosen as $\sigma(t) = c/\sqrt{\log(t+2)}$ for some $c
\geq c_0$, reffered as \emph{Continuous Diminishing Diffusion} (CDD, also known
as \emph{Simulated Annealing}), $x(t)$ converges to the set of global minimizers
under appropriate conditions on $f$. In \cite{Geman:1986js}, the objective
function is defined on a compact set, while the assumption is lifted in
\cite{Chiang:1987ip}. Other choices of the diffusion strength $\sigma(t)$ are
discussed in several articles.  The properties with large $\sigma(t)$
is discussed in \cite{Yin:2006fm}. More recently, a method called intermittent diffusion (ID) has been proposed in  \cite{CHOW:2013jm}, where a piecewise constant diffusion strength $\sigma(t) = \sum^N_{i=1}\sigma_i I_{[S_i,S_i+T_i]}(t)$ is considered. In other words, this method essentially considers to alternatively update variables between gradient descent and noisy gradient descent.  It has been shown in \cite{CHOW:2013jm} the global convergence ID and its effectiveness in specific problems.


To the best of our knowledge, only a few articles apply the SDE method to
constrained problems.   Problems with linear constrained is discussed in
\cite{Parpas:2006jq}. Portfolio selection having higher order moments with selected
constraints is studied in \cite{Maringer:2007jm}. The
authors in \cite{Parpas:2009hy} apply the method to robust chance constraint
problems and a class of minimax problems is solved in \cite{Parpas:2008hu}. In
\cite{Parpas:2009hx},  problems with equality constraints are solved but no theoretical validation is provided that the constraints can be preserved using the proposed SDE.

\subsection{Main Results}
In order to find the global minimizers of orthogonality constrained problems, it is natural to consider a generalization of the diffusion methods based on \cref{eq:sde:rn} in Euclidean space to problems with orthogonality constraints \cref{eq:prob:main}. Our strategy is a combination of the CDD and ID, which leads to an optimization procedure that alternatively apply the diminishing diffusion and the deterministic local solver mentioned in \cref{subsec:OptStfMfd}. We refer this procedure as an intermittent diminishing diffusion on manifold (IDDM). One crucial step of IDDM is to explore an computational tractable method to the SDE on the Stiefel manifold, which can be symbolically written as follows:
\begin{equation}
	\Diff X(t) = -\nabla_\Mfd \Fun(X(t))\Diff t + \sigma(t) \circ \Diff B_\Mfd (t),
	\label{eq:sde:stfmfd}
\end{equation}
where $\nabla_\Mfd$ and $B_\Mfd$ stand for gradient and Wiener process on manifold, respectively. One of the major challenges of using the above equation on the Stiefel manifold is the lack of global parameterization of the manifold, which make the numerical computation not straightforward to generate Wiener process on the Stiefel manifold. On the other hand, $\Mfd_{n,p}$ is an embedding manifold in $\Real^{n,p}$, whose embedding coordinates can be used to design an extrinsic form of the above SDE. In order to make use of \cref{eq:sde:stfmfd} on numerical optimization, we propose an extrinsic presentation to facilitate numerical work. Our idea is to project the Brownian motion in the ambient space to the tangent space of $\Mfd_{n,p}$. 
Based on this idea, we have theoretically validate the proposed procedure of IDDM for orthogonality constrained problems. More specifically, we have established the following results:
\begin{enumerate}
\item We theoretically show that the proposed extrinsic form is in fact generating feasible path constrained on $\Mfd_{n,p}$. 
\item We also validate that the proposed method of projection Brownian motion in
  $\Real^{n,p}$ to the tangent space of $\Mfd_{n,p}$ is an extrinsic form of the
  Brownian motion on the $\Mfd_{n,p}$.
\item We further propose a numerical-efficient scheme to solve the proposed extrinsic equation and theoretically validate the half-order convergence of the scheme. 
\item We also provide theoretical global convergence analysis of the proposed method, which is a consequence that the proposed extrinsic form satisfies the associated Fokker-Planck equation on the manifold. 
\item We numerically demonstrate that often only a few cycles of IDDM is needed to
  identify a better solution than the local algorithm for difficult problems with
  multiple local minimizers.
\end{enumerate}  

The rest of this paper is organized as follows. In section \ref{sec:SDEonStiefelMfd}, we propose an extrinsic form of the SDE \eqref{eq:sde:stfmfd} and discuss its well-posedness. We also show the proposed extrinsic diffusion term in fact provides the Brownnian motion on the Steifel manifold. Numerical scheme of solving the proposed SDE and its convergence is discussed in section \ref{sec:numericalscheme}. After that, we describe the proposed intermittent diminishing diffusion on manifold (IDDM) and show that IDDM  converges to global optimizers of the orthogonality constrained problems with probability almost equal to 1 in section \ref{sec:IDDM}. Numerically we demonstrate the effectiveness of the proposed method on several applications involving orthogonality constrained optimization in section \ref{sec:experiments}. Finally, we conclude our work in section \ref{sec:conclusion}. 


\section{SDE on Stiefel manifold}
\label{sec:SDEonStiefelMfd}

In this section, we propose an explicit representation of the SDE \eqref{eq:sde:stfmfd}. We also validate that the proposed explicit form is well-posed by showing that solutions of the explicit form stay on the Stiefel manifold with probability 1. We further show that the proposed method of projecting Brownian motion is a Brownian motion on the Stiefel manifold.

As we mentioned in the introduction, one crucial step of adapting SDE methods to the orthogonality constrained problems is how to design a computation tractable way of generating Brownian motion on the Stiefel manifold. Note that for any matrix $Z\in\Real^{n,p}$, we can use the following operator to project $Z$ to $\mathcal{T}_{X}\Mfd_{n,p}$.
\begin{equation}
\label{eqn:projection}
P:\Real^{n\times p} \rightarrow \mathcal{T}_{X}\Mfd_{n,p},\quad  Y\mapsto P_X(Z) = Z- \alpha XZ^{\top}X - \beta X X^{\top} Z
\end{equation}
where $\alpha = \sqrt{2}/2$, $\beta = 1-\sqrt{2}/2$. This motivates us to project the Brownian motion in the ambient space to the tangent space of $\Mfd_{n,p}$ based on this projection operator. Namely, we propose the extrinsic representation of the SDE \eqref{eq:sde:stfmfd} on Stiefel manifold as
	\begin{equation}
		\mathrm{d} X(t) = -\nabla_{\mathcal{M}} \mathcal{F}(X(t)) \mathrm{d} t + \sigma(t) \sum_{u=1}^n \sum_{v=1}^p P_{uv}(X(t)) \circ \mathrm{d} B_{uv}(t),
		\label{eq:sde:ext:stra}
	\end{equation}
	where $\{B_{uv}(t)\}$ is a series of (independent) one-dimensional standard Brownian motion, and $P_{uv}$ is defined by
	\begin{equation}
		\begin{aligned}
			P_{uv}(X) = E_{uv} - \alpha XE_{uv}^\top X - \beta XX^\top E_{uv}, \quad X \in \mathcal{M}_{n,p}, \\
			 u = 1,2,\dots,n, \quad v = 1,2,\dots,p.
		\end{aligned}
	\end{equation}
\subsection{Well-posedness of the extrinsic SDE}
\label{sec:diff:mfd}
There are several issues with respect to \cref{eq:sde:ext:stra} to be clarified. First, the definition of coefficients of drift term and diffusion term is restricted to the manifold, and thus a proper extension is needed in order to make it a well-posed SDE in Euclidean space $\mathbb{R}^{n\times p}$.  We first show that the SDE given by \cref{eq:sde:ext:stra} is well-posed. In other words, there is an equivalent extension to the euclidean space $\Real^{n\times p}$ that exists, lies on the manifold and gives a unique solution. We also expect that the solution will not leave the manifold so that the off-manifold coefficient will not impact the solution. The answers of all the above concerns are addressed in the following theorem. 


\begin{theorem}
	\label{thm:1}
	Let $V$ be an arbitrary smooth vector field on the Stiefel manifold $\Mfd_{n,p}$. Then
	\begin{enumerate}[(a)]
	\item There exists some smooth extensions of $V(X)$ and $P_{uv}(X)$ in $\Real^{n\times p}$ (denoted by $\tilde{V}(X)$ and $\tilde{P}_{uv}(X)$), which are globally Lipschitz. Hence, there exists a unique solution $X(t,w)$ for the extended SDE
		      \begin{equation}
		      	\mathrm{d} X(t) = \tilde{V}(X(t)) \mathrm{d} t + \sigma(t) \sum_{u=1}^n\sum_{v=1}^p \tilde{P}_{uv}(X(t)) \circ \mathrm{d} B_{uv}(t).
		      	\label{eq:sde:extended}
		      \end{equation}
		      in $\Real^{n\times p}$ once the extension is fixed.
		\item Let $X(t)$ be the solution of \cref{eq:sde:extended}, and then $X(t)$ almost surely stays on $\Mfd_{n,p}$ provided it originate on the manifold, i.e.,
		      \begin{equation}
		      	\Prob\{X(t) \in \Mfd_{n,p}|X(0) \in \Mfd_{n,p}\} = 1, \quad\forall~ t \geq 0.
		      	\label{eq:feasibility}
		      \end{equation}
		      In addition,  $X(t)$ does not leave its connected component in the case of $n = p$, i.e.,
		      \[
		      	\Prob\{\mathrm{det}(X(t)) = \mathrm{det}(X(0))|X(0) \in \Mfd_{n,n}\} = 1, \quad \forall~t\geq 0.
		      \]
		      The solution of \cref{eq:sde:extended} is unique regardless of the extension of $V$ and $P_{uv}$.
	\end{enumerate}
\end{theorem}

\begin{proof}
	\begin{enumerate}[(a)]
		\item Direct observation suggests that $\Mfd_{n,p}$ is a compact subset of $\Real^{n \times p}$, which makes it possible to construct a globally Lipschitz extension. The extension is not unique, and for example we can take 
		      \begin{equation}
		      	\begin{cases}
		      		\tilde{V}(X) := \zeta_\varepsilon(\|X^\top X - I_p\|_2^2)V(\mathcal{Q}(X)) ,                                                                                           \\
					\tilde{P}_{uv}(X) :=  \zeta_\varepsilon(\|X^\top X-I_p\|_2^2)\left(E_{uv}-\alpha XE_{uv}^\top X - \beta XX^\top E_{uv}\right),
		      	\end{cases}
		      	\label{eq:extension}
		      \end{equation}
		      where $\mathcal{Q}(X)$ indicates the $n$-by-$p$ matrix from reduced $QR$ decomposition of $X$ (here we follow the convention that the diagonal entries of upper triangular $R$ are non-negative). $\zeta_\varepsilon$ is a $C_0^\infty([0,+\infty))$ mollifier satisfying
		      \begin{equation}
		      	\zeta_\varepsilon([0,\varepsilon]) \equiv 1, \quad \zeta_\varepsilon([2\varepsilon,+\infty)) \equiv 0,
		      	\label{eq:mollifier}
		      \end{equation}
		      where $\varepsilon$ is a given positive constant with $\varepsilon < 1/2$. Under this condition one can show that both $\tilde{V}$ and $\tilde{P}_{ij}$ are globally Lipschitz. The existence and uniqueness of \cref{eq:sde:extended} follow directly from the existence and uniqueness theorem of general SDE (see Theorem 5.2.1 of \cite{oksendal:2003jp} for example).
		\item The general feasibility results \cref{eq:feasibility} can be derived by viewing \cref{eq:sde:extended} as a process driven by $\Real^{np+1}$-valued semimartingale $Z(t) = (t,\sigma(t)B_{ij}(t))$ and applying Proposition 1.2.8 of \cite{Hsu:2001tz}.  The special case of $n = p$ can be treated similarly but viewing two connected components as two separate manifolds instead. The uniqueness can be referred to Theorem 1.2.9 of \cite{Hsu:2001tz}.
	\end{enumerate}
\end{proof}

In view of the uniqueness result we can specify that the extension of $\tilde{V}$ and $\tilde{P}_{uv}$ is given by \cref{eq:extension} to facilitate further discussion. Sometimes it would be more convenient to analyze the Ito version of \cref{eq:sde:extended}, which can be derived from the following transformation property between Ito SDE and Stratonovich SDE in the Euclidean space. 
\begin{lemma}[\cite{oksendal:2003jp}]
	\label{lemma:transformation}
	The coresponding Ito version of Stratonovich system
\begin{equation*}
		\mathrm{d} X_{\eta}(t) = h_{\eta}(X(t),t)\mathrm{d} t + \sum_{\lambda}H_{\eta\lambda}(X(t),t)\circ \mathrm{d} B_\lambda(t)
		\end{equation*}
	is given by
	\begin{equation}
		\mathrm{d} X_{\eta}(t) = \left[h_{\eta}(X(t),t) + \frac{1}{2} \sum_{\lambda}\left(\sum_{\mu}\frac{\partial H_{\eta\lambda}}{\partial X_\mu}H_{\mu \lambda}\right)\right]\mathrm{d} t + \sum_{\lambda}H_{\eta\lambda}(X(t),t) \mathrm{d} B_\lambda(t).
		\label{eq:lemma:transform}
	\end{equation}
\end{lemma}

Based on this lemma, we can derive the Ito version of the Stratonovich SDE \cref{eq:sde:extended} described in the following theorem.
\begin{theorem}
	\label{thm:ito}
	The corresponding Ito version of Stratonovich SDE \cref{eq:sde:extended} on $\Mfd_{n,p}$  (with feasible initial point $X(0)\in \Mfd_{n,p}$) is given by
	\begin{equation}
		\begin{aligned}
			\mathrm{d} X(t) & = \left(V(X(t)) -  \frac{n-1}{2} \sigma^2(t)X(t)\right)\mathrm{d}t                                                                                                 \\
			                & + \sigma(t) \sum_{u=1}^n\sum_{v=1}^p \left( E_{uv}-\alpha XE_{uv}^\top X - \beta XX^\top E_{uv} \right)\mathrm{d} B_{uv}(t).
			\label{eq:sde:extended:ito}
		\end{aligned}
	\end{equation}
	Here we omit the discussion of definition of parameters outside the manifold.
\end{theorem}
\begin{proof} 
Write \cref{eq:sde:extended} coordinate-wise as
	\[
		\begin{aligned}
			\mathrm{d} X_{ij}(t) = & \tilde{V}_{ij}(X) \mathrm{d} t + \sigma(t)\sum_{u=1}^n\sum_{v=1}^p [\zeta_\varepsilon (\|X^\top X-I_p\|_2^2)            \\
			                       & (\delta_{iu}\delta_{jv}-\alpha X_{iv}X_{uj} - \beta \sum_{w=1}^p X_{iw}X_{uw}\delta_{jv}) \circ \mathrm{d} B_{uv}(t) ].
		\end{aligned}
	\]
	Applying \cref{lemma:transformation} by viewing the index $\eta = (i,j)$ and $\lambda = (u,v)$, one can show that the $(i,j)$-th entry of the additional drift term is given by
		\begin{align}
			   & \frac{1}{2}\sigma^2(t) \sum_{u,v,s,t} \left[\partial_{st}(-\alpha X_{iv}X_{uj}-\beta\sum_{w=1}^p X_{iw}X_{uw}\delta_{jv}) \right.   \nonumber  \\
			   & \cdot \left.(\delta_{us}\delta_{vt}-\alpha X_{sv}X_{ut}-\beta\sum_w^{p} X_{sw}X_{uw}\delta_{vt}) \right]      \nonumber                                                                         \\
			=~ & \frac{1}{2}\sigma^2(t) \sum_{u,v,s,t} [(-\alpha \delta_{is}\delta_{vt}X_{uj}-\alpha\delta_{us}\delta_{jt}X_{iv} -\beta\delta_{is}\delta_{jv}X_{ut} - \beta\delta_{jv}\delta_{us}X_{it})  \nonumber \\
			   & \cdot (\delta_{us}\delta_{vt}-\alpha X_{sv}X_{ut}-\beta\sum_{w=1}^{p} X_{sw}X_{uw}\delta_{vt})] \label{eqn:thm3_1}\\			
			   =~ & \frac{1}{2}\sigma^2(t)[(2\alpha^2+\beta^2+\alpha\beta-\beta)-(\alpha+\beta)n+(\beta^2+3\alpha\beta-\alpha)p]X_{ij} \nonumber \\
			=~    & -\frac{(n-1)}{2}\sigma^2(t)X_{ij}. \nonumber
		\end{align}
	Here we omit the discussion of coefficients off the manifold as the
    derivation of the mollifier $\eta_\varepsilon$ will not affect the
    on-manifold result due to the hypothesis of \cref{eq:mollifier}. In
    addition, the second last equality of the above derivation is provided by
    simply expanding each item of \eqref{eqn:thm3_1} using the facts
    $\sum_{u}X_{ui}X_{uj} = \delta_{ij}$ and $\sum_{uv} X_{uv}^2 = p$. More
    specifically, we summarize products among pairs in the following table:
\begin{center}
\vspace{0.3cm}
\begin{tabular}{|c|c|c|c|}
\hline
$\times$ & $\delta_{us}\delta_{vt}$ & 
$-\alpha X_{sv}X_{ut}$ & 
$-\beta\sum_{w=1}^{p} X_{sw}X_{uw}\delta_{vt}$ \\
\hline
 $\displaystyle -\alpha \sum_{u,v,s,t}  \delta_{is}\delta_{vt}X_{uj}$ & $ -\alpha p X_{ij} $ & $\alpha^2X_{ij}$ &$ \alpha\beta p X_{ij}$ \\
 \hline
$\displaystyle-\alpha \sum_{u,v,s,t} \delta_{us}\delta_{jt}X_{iv}$ & $ -\alpha n X_{ij}$ & $ \alpha^2X_{ij}$ & $\alpha\beta p X_{ij}$ \\
\hline
$\displaystyle -\beta \sum_{u,v,s,t} \delta_{is}\delta_{jv}X_{ut}$ & $-\beta X_{ij}$ & $\alpha\beta p X_{ij}$ & $\beta^2X_{ij}$ \\
\hline
$\displaystyle - \beta \sum_{u,v,s,t} \delta_{jv}\delta_{us}X_{it}$ & $-\beta nX_{ij} $& $\alpha\beta X_{ij}$ & $\beta^2 p X_{ij}$\\
\hline
\end{tabular}
\end{center}
\end{proof}
\begin{remark}
	{\rm
		Using the projection operator defined in \eqref{eqn:projection}, we can simplify the notation by writing the diffusion term in short as
		\begin{equation}
			\sum_{u=1}^n\sum_{v=1}^p \left(E_{uv}-\alpha XE_{uv}^\top X - \beta XX^\top E_{uv}\right) \mathrm{d} B_{uv}(t) = P_X(\Diff B(t)).
		\end{equation}
		For example, the above Ito SDE \cref{eq:sde:extended:ito} can be simplified as
		\begin{equation}
			\Diff X(t) = \left(V(X(t)) -  \frac{n-1}{2} \sigma^2(t)X(t)\right)\mathrm{d}t + \sigma(t)P_X(\Diff B(t)).
		\end{equation}
		We will follow this convention throughout the paper.
	}
\end{remark}

\subsection{Laplace-Beltrami Operator on Canonical Stiefel Manifold}
\label{subsec:LB}
It remains to show that the diffusion term of extrinsic SDE \cref{eq:sde:extended} is the Brownian motion on the Stiefel manifold. Before considering that, we first provide an extrinsic representation of the Laplace-Beltrami (LB) operator on the Stiefel manifold.

\begin{theorem}[Extrinsic form of the LB operator on $\Mfd_{n,p}$]\label{thm:lb}
	The LB operator at $X$ on $\mathcal{M}_{n,p}$ (endowed with the canonical metric) is given by
	\begin{equation}
		\Delta_{\mathcal{M}_{n,p}} = \sum_{i=1}^n\sum_{j=1}^p \partial_{ij}^2 - \sum_{i,u=1}^n\sum_{j,v=1}^p X_{iv}X_{uj}\partial_{ij}\partial_{uv}  - (n-1)\sum^n_{i=1}\sum^p_{j=1} X_{ij}\partial_{ij}.
		\label{eq:lb}
	\end{equation}
\end{theorem}
We calculate the LB operator using 
the trace of the Hessian operator along an orthonormal basis in the tangent space $\mathcal{T}_X{\mathcal{M}_{n,p}}$. First, we provide an orthonormal basis in the following lemma.
\begin{lemma}[Orthonormal basis of $\mathcal{T}_X\mathcal{M}_{n,p}$] \label{lemma:orth}
	Let $Q$ to be an extended orthogonal matrix of $X$, such that $Q\in \Real ^{n\times n}$, $Q^\top Q = I_n$ and $Q = [X,X_\perp]$. An orthonormal basis of $\mathcal{T}_X\Mfd_{n,p}$ is given by
	\[
		\begin{aligned}
			U_{ij} = & Q(E_{ij}-E_{ji}), & \quad i<j\leq p, \\
			U_{ij} = & QE_{ij},          & \quad i>p,
		\end{aligned}
	\]
	where $E_{ij}$ is the element matrix in $\Real^{n\times p}$. The set of the basis is denoted by $\Lambda$.
\end{lemma}
\begin{proof}[Proof of \cref{lemma:orth}]
	To simplify the notation, we set $\tilde{E}_{ij} = E_{ij}-E_{ji},(i,j\leq p)$ or $\tilde{E}_{ij} = E_{ij},(i>p)$. 
The orthogonality of $Q$ indicates that $\Lambda$ is linear independent. It can be easily calculated that $\Lambda$ has $np-p(p+1)/2$ elements so that $\Lambda$ spans $\mathcal{T}_X\Mfd_{n,p}$. We next calculate the inner product as
	\[
		\begin{aligned}
			g^c(U_{ij},U_{kl})   & =\mathrm{tr}\left(\tilde{E}^\top_{ij} \mathrm{diag}\left\{\frac{1}{2}I_p,I_n\right\}\tilde{E}_{kl}\right) = 0, \quad \text{if}~ (i,j)\neq(k,l)~\text{and}~(j,i)\neq(k,l);\\
			g^c(U_{ij},U_{ij}) & = \mathrm{tr}\left(\tilde{E}^\top_{ij} \mathrm{diag}\left\{\frac{1}{2}I_p,I_n\right\}\tilde{E}_{ij}\right) =1, \quad \text{if}~ i \neq j;                    \\
		\end{aligned}
	\]
	The results above indicate that $\Lambda$ is the set of orthonormal basis of $\mathcal{T}_X\Mfd_{n,p}$.
\end{proof}
From the orthonormal basis we can calculate the LB operator $\Delta_{\Mfd_{n,p}}$.
\begin{proof}[Proof of \cref{thm:lb}]
	It has been shown in \cite{Edelman:1998ei} that
	\begin{equation}
		\begin{aligned}
			\nabla^2_{\Mfd_{n,p}} \Fun(Z_1,Z_2)= & \nabla^2_{E}\Fun(Z_1,Z_2)+\frac{1}{2}\tr(G^\top Z_1X^\top Z_2+X^\top Z_1 G^\top Z_2)             \\
			-                            & \frac{1}{2}\tr((X^\top G +G^\top X)Z^\top_1(I-XX^\top)Z_2),\ Z_1,Z_2\in \mathcal{T}_X\Mfd_{n,p},
		\end{aligned}
		\label{eq:hess:c}
	\end{equation}
	where $\nabla^2_{E}$ is the Hessian operator in the Euclidean space $\Real^{n\times p}$ and $G_{ij}=\partial_{ij}\Fun$.

	From the orthonormal basis we obtain
	\[
		\begin{aligned}
			\nabla^2_{\Mfd_{n,p}} \Fun(U_{ij},U_{ij}) = & \nabla^2_E\Fun(Q\tilde{E}_{ij},Q\tilde{E}_{ij})+\tr(G^\top Q\tilde{E}_{ij}X^\top Q\tilde{E}_{ij})                \\
			-                                   & \frac{1}{2}\tr((X^\top G+G^\top X)\tilde{E}^\top_{ij}Q^\top(I-XX^\top)Q\tilde{E}_{ij})                           \\
			=                                   & \nabla^2_E\Fun(Q\tilde{E}_{ij},Q\tilde{E}_{ij})+\tr(G^\top Q\tilde{E}_{ij}[I_p,0_{p\times (n-p)}]\tilde{E}_{ij}) \\
			-                                   & \frac{1}{2}\tr((X^\top G+G^\top X)\tilde{E}^\top_{ij}\mathrm{diag}\{0_p,I_{n-p}\}\tilde{E}_{ij}).
		\end{aligned}
	\]
	Hence, we have
	\[
		\begin{aligned}
			\Delta_{\Mfd_{n,p}} \Fun = & \sum_{U_{ij}\in \Lambda}\nabla^2_{\Mfd_{n,p}}\Fun(U_{ij},U_{ij})                                                                             \\
			=                  & \sum_{i>p}[\nabla^2_E\Fun(QE_{ij},QE_{ij})+0-\frac{1}{2}\text{tr}((X^\top G +G^\top X)E^p_{jj})]                                     \\
			+                  & \sum_{i<j\leq p}[\nabla^2_E\Fun(Q\tilde{E}_{ij},Q\tilde{E}_{ij})+\tr(G^\top Q\tilde{E}_{ij}[I_p,0_{p\times (n-p)}]\tilde{E}_{ij})-0] \\
			=                  & \sum_{i>p}[\nabla^2_E\Fun(QE_{ij},QE_{ij})-(G^\top X)_{jj}]                                                                          \\
			+                  & \sum_{i<j\leq p}[(\nabla^2_E\Fun(QE_{ij},QE_{ij})+\nabla^2_E\Fun(QE_{ji},QE_{ji})-2\nabla^2_E\Fun(QE_{ij},QE_{ji}))                  \\
			-                  & ((G^\top X)_{ii}+(G^\top X)_{jj})]                                                                                                   \\
			=                  & \sum_{i,j}\nabla^2_EF(QE_{ij},QE_{ij})-\sum_{i,j\leq p}\nabla^2_EF(QE_{ij},QE_{ji})-(n-1)\text{tr}(G^\top X).
		\end{aligned}
	\]
	Orthogonal $Q$ indicates that the linear transformation $V\in \Real^{n\times
    p}\mapsto QV \in \Real^{n\times p}$ is orthogonal (under Euclidean metric).
    Therefore,
	\[
		\begin{aligned}
			\Delta_{\Mfd_{n,p}} \Fun = & \sum_{i,j}\nabla^2_EF(QE_{ij},QE_{ij})-\sum_{i,j\leq p}\nabla^2_EF(QE_{ij},QE_{ji})-(n-1)\text{tr}(G^\top X) \\
			=                  & \Delta_E\Fun-\sum_{i,j\leq p}\nabla^2_E\Fun(XE^p_{ij},XE^p_{ji})-(n-1)\tr(G^\top X),
		\end{aligned}
	\]
	 where $E^p_{ij}$ is the element matrix in $\Real^{p\times p}$. With some expansion and mark changing, we obtain
	\[
		\Delta_{\Mfd_{n,p}}=\sum_{i=1}^n\sum_{j=1}^p \partial_{ij}^2 - \sum_{i,u=1}^n\sum_{j,v=1}^p X_{iv}X_{uj}\partial_{ij}\partial_{uv}  - (n-1)\sum^n_{i=1}\sum^p_{j=1} X_{ij}\partial_{ij},
	\]
	which completes the proof.
\end{proof}



\subsection{Extrinsic formulation of Brownian motion on the Stiefel Manifold}
We now  show that the diffusion term introduced in \cref{eq:sde:extended} is exactly the $\Mfd$-valued Brownnian motion driven by half of the Laplace-Beltrami operator. We state the result in the following theorem.
\begin{theorem}
	Suppose that $W(t)$ is the solution of the following SDE
	\begin{equation}
		\mathrm{d} W(t) = \sum_{u=1}^n\sum_{v=1}^p \left( E_{uv}-\alpha WE_{uv}^\top W - \beta WW^\top E_{uv} \right) \circ \mathrm{d} B_{uv}(t).
		\label{eq:bm:sde}
	\end{equation}
	Then $W(t)$ is driven by half of Laplacian-Beltrami operator
    $\Delta_{\Mfd_{n,p}}$ on Stiefel Manifold, i.e.,
	\[
		\frac{1}{2}\Delta_{\Mfd_{n,p}} \varphi(W(t)) = \mathcal{L}\varphi(W(t)) := \lim_{t\rightarrow 0+} \frac{\Expect [\varphi(W(t))|W(0)=w_0]-\varphi(w_0)}{t}.
	\]
\end{theorem}
\begin{proof}
	From \cref{thm:ito}, the Ito version of \cref{eq:bm:sde} is
	\[
		\mathrm{d} W(t) = -\frac{n-1}{2} X\mathrm{d} t + \sum_{u=1}^n\sum_{v=1}^p  (E_{uv}-\alpha WE_{uv}^\top W -\beta WW^\top E_{uv} ) \mathrm{d} B_{uv}(t).
	\]
	The generator of $\varphi$ can be derived as \cite{oksendal:2003jp}
	\[
		\begin{aligned}
			\mathcal{L}\varphi & = -\frac{n-1}{2} \sum^n_{i=1}\sum^p_{j=1} X_{ij}\partial_{ij}\varphi + \frac{1}{2} \sum_{i,u,s}^n\sum_{j,v,t}^p \left( \delta_{is}\delta_{jt} - \alpha X_{it}X_{sj} - \beta \sum_w^p X_{iw}X_{sw}\delta_{tj} \right) \\
			             & \cdot\left(\delta_{us}\delta_{vt} -\alpha X_{ut}X_{sv} -\beta \sum_{z}^p X_{uz}X_{sz}\delta_{tv} \right) \partial_{ij}\partial_{uv} \varphi.
		\end{aligned}
	\]
We expand pairwise products in the second term of the above equation as follows: 
	\[
		\sum_{i,j,u,v,s,t} (\delta_{is}\delta_{jt})(\delta_{us}\delta_{vt})(\partial_{ij}\partial_{uv}\varphi) = \sum_{i,j} \partial^2_{ij}\varphi,
	\]
	\[
		\sum_{i,j,u,v,s,t}(\delta_{is}\delta_{jt})(-\alpha X_{ut}X_{sv})(\partial_{ij}\partial_{uv}\varphi) 
		= -\alpha \sum_{i,j,u,v} X_{uj}X_{iv} (\partial_{ij}\partial_{uv} \varphi), \text{\footnotesize (twice)}
	\]
	\[
		\sum_{i,j,u,v,s,t}(\delta_{is} \delta_{jt})(-\beta \sum_{z=1}^p X_{uz} X_{sz} \delta_{tv})(\partial_{ij}\partial_{uv} \varphi) = -\beta \sum_{i,j,u,v} X_{iv}X_{uv}( \partial_{ij}\partial_{uj} \varphi), \text{\footnotesize (twice)}
	\]
	\[
		\sum_{i,j,u,v,s,t}(-\alpha X_{it}X_{sj})(-\alpha X_{ut}X_{sv})(\partial_{ij}\partial_{uv}\varphi) = \alpha^2 \sum_{i,j,u,v}  X_{iv}X_{uv} (\partial_{ij}\partial_{uj}\varphi),
	\]
	\[
		\sum_{i,j,u,v,s,t}(-\alpha X_{it}X_{sj})(-\beta \sum_{z=1}^p X_{uz}X_{sz}\delta_{tv})(\partial_{ij}\partial_{uv}\varphi) = \alpha\beta \sum_{i,j,u,v} X_{iv}X_{uj} (\partial_{ij}\partial_{uv} \varphi), \text{\footnotesize (twice)}
	\]
	\[
		\begin{aligned}
			  & \sum_{i,j,u,v,s,t} (-\beta \sum_{w=1}^p X_{iw}X_{sw} \delta_{tj})(-\beta \sum_{z=1}^p X_{uz}X_{sz} \delta_{tv}) (\partial_{ij}\partial_{uv}\varphi)    \\
			= & ~ \beta^2 \sum_{i,j,u} (XX^\top XX^\top)_{iu} (\partial_{ij}\partial_{uj}\varphi) = \beta^2 \sum_{i,j,u,v} X_{iv}X_{uv} (\partial_{ij}\partial_{uj} \varphi).
		\end{aligned}
	\]
	Hence, we have
	\[
		\begin{aligned}
		 \mathcal{L} \varphi = & -\frac{n-1}{2} \sum^n_{i=1}\sum^p_{j=1} X_{ij}\partial_{ij}\varphi + \frac{1}{2}\left[\sum_{i,j}\partial_{ij}^2\varphi + (-2\alpha + 2\alpha\beta) \sum_{i,j,u,v} X_{iv}X_{uj} \partial_{ij}\partial_{uv} \varphi \right. \\
		& \left. + (-2\beta + \alpha^2 + \beta^2)(\sum_{i,j,u,v} X_{iv}X_{uv}\partial_{ij}\partial_{uj}\varphi) \right].
		\end{aligned}
	\]
	Substituting $\alpha = \sqrt{2}/2$ and $\beta = 1-\sqrt{2}/2$ we obtain
	\[
		\mathcal{L} \varphi = -\frac{n-1}{2} \sum^n_{i=1}\sum^p_{j=1} X_{ij}\partial_{ij}\varphi + \frac{1}{2}\left[\sum_{i,j}\partial_{ij}^2\varphi - \sum_{i,j,u,v} X_{iv}X_{uj} \partial_{ij}\partial_{uv} \varphi \right] = \frac{1}{2} \Delta_{\Mfd_{n,p}} \varphi.
	\]
\end{proof}
The next  corollary is a direct extension of the above theorem.
\begin{corollary}
	The Fokker-Planck Equation of \cref{eq:sde:ext:stra} is given by
	\begin{equation}
		\frac{\partial p}{\partial t} = -\nabla_{\Mfd_{n,p}}\cdot(p \nabla_{\mathcal{M}_{n,p}} \Fun ) + \frac{1}{2}\sigma^2(t)\Delta_{\Mfd_{n,p}} p,
		\label{eq:fpe}
	\end{equation}
	where $\nabla_{\Mfd_{n,p}}$, $\nabla_{\Mfd_{n,p}}\cdot$, $\Delta_{\Mfd_{n,p}}$ represent the gradient, divergence and Laplace-Beltrami operator on the Stiefel manifold endowed with canonical metric, respectively. 
\end{corollary}


\section{Numerical scheme of the SDE and its convergence}
\label{sec:numericalscheme}
We next provide a numerical scheme to solve the SDE \eqref{eq:sde:ext:stra}. Our idea is first projecting the random noise in the ambient space to the tangent space of the Stiefel manifold. After that, we apply the Cayley transformation similar as the method discussed in \cite{Wen:2012ga}. More precisely, we propose the following update scheme to solve the SDE \eqref{eq:sde:ext:stra}:
\begin{equation}
	\begin{cases}
		\displaystyle Z_k = -\delta_kG_k + \sigma_k(I_n-\beta Y_kY_k^\top)\delta B_k,                                                                  \\
		A_k = Z_kY_k^\top - Y_kZ_k^\top,                                                                                                                               \\
		\displaystyle
        Y_{k+1}=\left(I-\frac{A_k}{2}\right)^{-1}\left(I+\frac{A_k}{2}\right)Y_k.
        \\
	\end{cases}
	\label{eq:imp}
\end{equation}
In the case of $p=n$, we have a simpler form
\[
    Z_k = -\delta_k G_k + \alpha\sigma_k \delta B_k.
\]
In the spherical constrained case of $p=1$, we can show that
\[
	A_k = (-\delta_k G_k + \sigma_k \delta B_k) Y_k^\top - Y_k(-\delta_k G_k + \sigma_k \delta B_k)^\top.
\]

We point out that there is an efficient way to compute $Y_{k+1}$ in the case of $p<n/2$ or $p=1$ based on the Sherman-Morrison-Woodbury formula similar to the way discussed in \cite{Wen:2012ga}.

\begin{lemma}[\cite{Wen:2012ga}]
	\label{lm:faster}
	\begin{enumerate}[(1)]
		\item Rewrite $A_k=U_kV_k^\top$ for $U_k=[Z_k,Y_k]$ and $V_k=[Y_k,-Z_k]$. If $I-\frac{1}{2}V^\top_kU_k$ is invertible, then
		      \begin{equation}
		      	Y_{k+1}=Y_k+U_k\left(I-\frac{1}{2}V^\top_kU_k\right)^{-1}V^\top_kY_k.
		      \end{equation}
		\item For the vector case,
		      \begin{equation}
		      	Y_{k+1} = Y_k + \frac{Z_k}{1-(\frac{1}{2})^2(Z_k^\top Y_k)^2+(\frac{1}{2})^2Z_k^\top Z_k}-\frac{Z_k^\top Y_k - \frac{1}{2}((Z_k^\top Y_k)^2)+Z_k^\top Z_k}{1-(\frac{1}{2})^2(Z_k^\top Y_k)^2+(\frac{1}{2})^2Z_k^\top Z_k}Y_k.
		      \end{equation}
	\end{enumerate}
\end{lemma}

The numerical scheme can now be summarized in \cref{alg:ns}.
\begin{algorithm}
	\caption{Numerical Scheme of the SDE}
	\label{alg:ns}
	\begin{algorithmic}[1]
		\REQUIRE {Diffusion strength $\sigma(t)$, time discretization $t_0=\tau_0<\tau_1<\dots < \tau_K = T$, initial point $Y_0=X(t_0)$};
        \STATE Let $\delta_k=\tau_{k+1}-\tau_k$, $\sigma_k=\sigma(\tau_k)$ and $G_k=\nabla_E\Fun(Y_k)$ for simplification;
        \STATE Generate a series of $n$-by-$p$ independent random matrices $\{\delta B_k\}_{k=0}^{K-1}$ , the entries of which are independent $N(0,\delta_k)$ Gaussian variables;
        \FOR {$k=0:K-1$}
        \STATE Generate $Y_{k+1}$ from the update scheme \cref{eq:imp}.
        \ENDFOR
        \STATE We consider $Y_k$ to be an appropriate approximation of $X(\tau_k)$.
	\end{algorithmic}
\end{algorithm}

Now we state the strong convergence result. For simplicity, we state and prove the result in the case of constant $\sigma(t) \equiv \sigma_0$. Similar result can be proved with variational $\sigma(t)$ under trivial changes.
\begin{theorem}[Half Order Strong Convergence]
	Denote $X(T)$ as a solution of the SDE \eqref{eq:sde:ext:stra} and let
    $\delta = \max_{k}\{\delta_k\}$. Then there exists a positive constant $C = C(T)$ independent of $\delta$, as well as a constant $\delta_0 > 0$ such that
	\begin{equation}
		\mathbb{E}\|X(T) - Y_K\|_2^2 \leq C\delta, \quad \forall \delta \in
        (0,\delta_0).
	\end{equation}
\end{theorem}
\begin{proof}
	Without loss of generality, we suppose $t_0 = 0$. For $0 = t_0 \leq t \leq
    T$, we define
	\begin{equation}
		R(t) := \sup_{0 \leq s \leq t}\mathbb{E}\|X(s)-Y_{k_s}\|_F^2,
		\label{eq:def:residue}
	\end{equation}
	where $k_t$ is the largest integer $k$ for which $\tau_k$ does not exceed
    $t$, i.e.,
	\[
		k_t := \max\{k = 0,1,\dots,K: \tau_k \leq t\}.
	\]
	Rewriting \eqref{eq:sde:ext:stra} into an Ito integral form yields
	\[
		X(s) - X(0) = \int_0^s\left[-\nabla_{\mathcal{M}}f(X(\tau))-\frac{n-1}{2}\sigma_0^2 X(\tau)\right]\mathrm{d}~\tau + \sigma_0 \int_0^s \sum_{u,v} P_{uv}(X(\tau)) \mathrm{d} B(\tau).
	\]
	Substituting into \eqref{eq:def:residue} and applying the Schwarz inequality yields
	\begin{equation}
		\begin{aligned}
			     & R(t) =  \sup_{0 \leq s \leq t} \mathbb{E} \left\| \sum_{k=0}^{k_s - 1} (Y_{k+1}-Y_k) \right.                                                                                                                    \\
			     & \left.+ \int_{0}^s \left[\nabla_{\mathcal{M}}\Fun(X(\tau))+\frac{n-1}{2}\sigma_0^2 X(\tau) \right] \Diff \tau - \sigma_0\int_0^s P_{X(\tau)}(\Diff B(\tau)) \right\|_F^2                                                    \\
			\leq & ~ 7\sup_{0 \leq s \leq t} \left \{ \mathbb{E} \left\| \sum_{k=0}^{k_s-1} \left[\Expect (Y_{k+1}-Y_k|Y_k) - \delta_k\left(-\nabla_{\mathcal{M}}\Fun (Y_k) - \frac{n-1}{2}\sigma_0^2 Y_k \right) \right] \right\|_F^2 \right. \\
			     & + \Expect \left\| \sum_{k=0}^{k_s-1} \left[ Y_{k+1}-Y_k - \Expect(Y_{k+1}-Y_k|Y_k)-\sigma_0 P_{Y_k}(\delta B_k)  \right] \right\|_F^2                                                                           \\
			     & + \Expect \left\|\sum_{k=0}^{k_s-1}\left[ \int_{\tau_k}^{\tau_{k+1}}[\nabla_{\mathcal{M}}\Fun(X(\tau))]\Diff \tau -\delta_k\nabla_{\mathcal{M}}\Fun(Y_{\tau_k})\right]\right\|_F^2                                                    \\
			     & + \Expect \left\|\sum_{k=0}^{k_s-1}\left[ \int_{\tau_k}^{\tau_{k+1}}\left(\frac{n-1}{2}\sigma_0^2X(\tau)\right)\Diff \tau -\delta_k\left(\frac{n-1}{2}\sigma_0^2Y_{\tau_k}\right)\right]\right\|_F^2            \\
			     & + \Expect \left\|\sigma_0\sum_{k=0}^{k_s-1}\left[\int_{\tau_k}^{\tau_{k+1}}P_{X(\tau)}(\Diff B(\tau)) - P_{Y_k}(\delta B_k)\right] \right\|_F^2                                                                 \\
			     & + \Expect \left\|\int_{\tau_{k_s}}^{s}\left(\nabla_{\mathcal{M}}\Fun(X(\tau))+\frac{n-1}{2}\sigma_0^2X(\tau)\right) \Diff \tau \right\|_F^2                                                                                \\
			     & + \left. \Expect \left\|\sigma_0\int_{\tau_{k_s}}^{s}P_{X(\tau)}(\Diff B(\tau))\right\|_F^2 \right\}
		\end{aligned}
		\label{eq:seventerm}
	\end{equation}
	We next analyze the seven terms of \eqref{eq:seventerm} in order. Iterating \eqref{eq:imp} repeatedly yields
	\begin{equation}
		Y_{k+1} - Y_k = A_kY_k + \frac{1}{2}A_k^2Y_k +
        \frac{1}{8}A_k^3\left(Y_k+Y_{k+1}\right).
	\end{equation}
	Direct calculus shows that
	\begin{equation}
		\begin{aligned}
			A_kY_k & = Z_k - Y_kZ_k^\top Y_k                                                                                                    \\
			       & = -\delta_k(G_k-Y_kG_k^\top Y_k) + \sigma_0[(I-\beta Y_kY_k^\top)\delta B_k - Y_k\delta B_k^\top (I-\beta Y_kY_k^\top)Y_k] \\
			       & = -\delta_k(\nabla_{\mathcal{M}}\Fun(Y_k)) + \sigma_0(\delta B_k - \alpha Y_k\delta B_k^\top Y_k - \beta Y_kY_k^\top \delta B_k)       \\
			       & = -\delta_k(\nabla_{\mathcal{M}}\Fun(Y_k)) + \sigma_0 P_{Y_k}(\delta B_k),
		\end{aligned}
	\end{equation}
	and
	\begin{equation}
		\begin{aligned}
			A_k^2Y_k & = (Z_kY_k^\top - Y_kZ_k^\top)A_k Y_k                                                                                                   \\
			         & = [\delta_k (-G_kY_k^\top + Y_kG_k^\top)+ \sigma_0((I-\beta Y_kY_k^\top)\delta B_kY_k^\top -Y_k\delta B_k^\top (I-\beta Y_kY_k^\top))] \\
			         & \cdot [-\delta_k \nabla_{\mathcal{M}} \Fun(Y_k)+ \sigma_0 P_{Y_k}(\delta B_k)].
		\end{aligned}
	\end{equation}
	We claim that
	\begin{equation}
		\Expect\{[(I-\beta Y_kY_k^\top)\delta B_kY_k^\top - Y_k\delta B_k^\top(I-\beta Y_kY_k^\top)](P_{Y_k}(\delta B_k))|Y_k\} = -(n-1) \delta_k Y_k.
		\label{eq:claimexpect}
	\end{equation}
	In fact, we can show that
	\begin{equation}
		\begin{aligned}
			         & \Expect\{[(I-\beta Y_kY_k^\top)\delta B_kY_k^\top - Y_k\delta B_k^\top(I-\beta Y_kY_k^\top)](P_{Y_k}(\delta B_k))|Y_k\}   \\
			=~       & \Expect\{(\delta B_kY_k^\top-\beta Y_kY_k^\top\delta B_kY_k^\top-Y_k\delta B_k^\top+\beta Y_k\delta B_k^\top Y_kY_k^\top) \\
			\cdot ~ & (\delta B_k-\alpha Y_k \delta B_k^\top Y_k-\beta Y_kY_k^\top\delta B_k)|Y_k\}.
		\end{aligned}
	\end{equation}
	Let $Q_k=[Y_k,Y_{k}^\perp]$ and $N_k = Q_k^\top \delta B_k$, and one can show that the entries of $N_k$ are independent $N(0,\delta_k)$ variables. By substituting $\delta B_k=Q_k N_k$ back into the above equation and expanding the corresponding terms, we have
	\[
		\Expect\{Q_kN_kY^\top_kQ_kN_k|Y_k\}=\Expect\{Q_kN_k[I_p,0]N_k|Y_k\}=\delta_kQ_k\left(
		\begin{matrix}
			I_p               \\
			0
		\end{matrix}
		\right)=\delta_kY_k,
	\]
	\[
		-\alpha \Expect\{Q_kN_kY^\top_k Y_kN^\top_kQ_k^\top Y_k|Y_k\}=-\delta_k\alpha p Q_kI_nQ_k^\top Y_k=-\delta_k\alpha pY_k,
	\]
	\[
		-\beta\Expect\{Q_kN_kY^\top_k Y_kY_k^\top Q_kN_k|Y_k\}=-\beta\Expect\{Q_kN_kI_p[I_p,0]N_k|Y_k\}=-\delta_k\beta Y_k,
	\]
	\[
		-\beta\Expect\{Y_kY_k^\top Q_kN_kY_k^\top Q_kN_k|Y_k\}=-\beta\Expect\{Y_k[I_p,0]N_k[I_p,0]N_k|Y_k\}=-\delta_k\beta Y_k,
	\]
	\[
		\alpha\beta\Expect\{Y_kY_k^\top Q_kN_kY_k^\top Y_kN^\top_kQ_k^\top Y_k|Y_k\}=\alpha\beta\Expect\{Y_k[I_p,0]N_kN_k^\top\left(
		\begin{matrix}
			I_p \\
			0
		\end{matrix}
		\right)|Y_k\}=\delta_k\alpha\beta pY_k,
	\]
	\[
		\beta^2\Expect\{Y_kY_k^\top Q_kN_kY_k^\top Y_kY_k^\top Q_kN_k|Y_k\}=\beta^2\Expect\{Y_k[I_p,0]N_k[I_p,0]N_k|Y_k\}=\delta_k\beta^2 Y_k,
	\]
	\[
		-\Expect\{Y_kN_k^\top Q_k^\top Q_kN_k|Y_k\}=-\delta_knY_k,
	\]
	\[
		\alpha\Expect\{Y_kN_k^\top Q_k^\top Y_kN^\top_kQ_k^\top Y_k|Y_k\}=\alpha\Expect\{Y_kN_k^\top\left(
		\begin{matrix}
			I_p \\
			0
		\end{matrix}
		\right)N_k^\top\left(
		\begin{matrix}
			I_p \\
			0
		\end{matrix}
		\right)|Y_k\}=\delta_k\alpha Y_k,
	\]
	\[
		\beta\Expect\{Y_kN_k^\top Q_k^\top Y_kY_k^\top Q_kN_k|Y_k\}=\beta\Expect\{Y_kN_k^\top \mathrm{diag}\{I_p,0_{n-p}\}N_k|Y_k\}=\delta_k\beta p Y_k,\ \text{\footnotesize (twice)},
	\]
	\[\begin{aligned}
		&-\alpha\beta\Expect\{Y_kN_k^\top Q_k^\top Y_kY_k^\top Y_kN^\top_kQ_k^\top Y_k|Y_k\}\\
		=&-\alpha\beta\Expect\{Y_kN_k^\top\left(
		\begin{matrix}
			I_p \\
			0
		\end{matrix}
		\right)N_k^\top\left(
		\begin{matrix}
			I_p \\
			0
		\end{matrix}
		\right)|Y_k\}=-\delta_k\alpha\beta Y_k,\\
		&-\beta^2\Expect\{Y_kN_k^\top Q_k^\top Y_kY_k^\top Y_kY_k^\top Q_kN_k|Y_k\}\\
		=&-\beta^2\Expect\{Y_kN_k^\top \mathrm{diag}\{I_p,0_{n-p}\}N_k|Y_k\}=-\delta_k\beta^2 p Y_k.
		\end{aligned}
	\]
	Taking sum of the above terms yields \cref{eq:claimexpect}. A direct corollary of \cref{eq:claimexpect} is
	\begin{equation}
		\Expect \left\|\Expect (Y_{k+1}-Y_k|Y_k) - \delta_k\left(-\nabla_{\mathcal{M}} \Fun (Y_k) -\frac{n-1}{2}\sigma_0^2Y_k\right) \right\|_F^2 \leq C_1'\delta_k^3.
	\end{equation}
	Hence, we can derive an estimation of the first term in \cref{eq:seventerm} as
	\begin{equation}
		\begin{aligned}
			     & ~ \mathbb{E} \left\| \sum_{k=0}^{k_s-1} \left[\Expect (Y_{k+1}-Y_k|Y_k) - \delta_k\left(-\nabla_{\mathcal{M}}\Fun (Y_k) - \frac{n-1}{2}\sigma_0^2 Y_k \right) \right] \right\|_F^2                                         \\
			\leq & ~ \delta\left(\sum_{k=0}^{k_s-1} \delta_k\right) \sum_{k=0}^{k_s-1}\frac{1}{\delta_k^2} \Expect \left\|\Expect (Y_{k+1}-Y_k|Y_k) - \delta_k \left(- \nabla_{\mathcal{M}} \Fun (Y_k) - \frac{n-1}{2}\sigma_0^2Y_k\right) \right\|_F^2 \\
			\leq & ~ \delta \left(\sum_{k=0}^{k_s-1}\delta_k\right)  \left(\sum_{k=0}^{k_s-1}C'_1 \delta_k\right) \leq C_1 \delta.
		\end{aligned}
		\label{eq:term:1}
	\end{equation}
	The first inequality of \cref{eq:term:1} is due to Cauchy-Schwartz inequality and $\delta = \max\{\delta_k\}$. 
	
	The second term of \cref{eq:seventerm} can be evaluated in view that all the cross-product terms vanish under the Frobenius norm:
	\begin{equation}
		\begin{aligned}
			     & ~\Expect \left\| \sum_{k=0}^{k_s-1} \left[ Y_{k+1}-Y_k - \Expect(Y_{k+1}-Y_k|Y_k)-\sigma_0 P_{Y_k}(\delta B_k)  \right] \right\|_F^2 \\
			=    & \sum_{k=0}^{k_s-1} \Expect \|[ Y_{k+1}-Y_k - \Expect(Y_{k+1}-Y_k|Y_k)-\sigma_0 P_{Y_k}(\delta B_k) ] \|_F^2                          \\
			\leq & \sum_{k=0}^{k_s-1} (\Expect \|A_kY_k -\sigma_0 P_{Y_k}(\delta B_k)\|_F^2 + C_2'\delta_k^2)                                           \\
			\leq & \sum_{k=0}^{k_s-1} C_2 \delta_k^2 \leq C_2\delta.
		\end{aligned}
		\label{eq:term:2}
	\end{equation}
	
	The third term of \cref{eq:seventerm} can be estimated in view of the smoothness of $\nabla_{\mathcal{M}} \Fun$:
	\begin{equation}
		\begin{aligned}
			     & ~\Expect \left\|\sum_{k=0}^{k_s-1}\left[ \int_{\tau_k}^{\tau_{k+1}}[-\nabla_{\mathcal{M}}\Fun(X(\tau))]\Diff \tau + \delta_k\nabla_{\mathcal{M}}\Fun(Y_{k_\tau})\right]\right\|_F^2 \\
			=    & ~ \Expect \left\|\int_0^{\tau_{k_s}} \left[\nabla_{\mathcal{M}} \Fun (Y_{k_\tau}) -\nabla_{\mathcal{M}} \Fun(X(\tau)) \right]\Diff \tau \right\|_F^2                                \\
			\leq & ~ T \int_0^{\tau_{k_s}} \Expect \| [\nabla_{\mathcal{M}} \Fun(Y_{k_\tau}) -\nabla_{\mathcal{M}} \Fun(X(\tau))] \|_F^2 \Diff \tau                                                    \\
			\leq & ~ TC_3' \int_0^{\tau_{k_s}} \Expect \|X(\tau)-Y_{k_\tau}\|_F^2 \Diff \tau                                                                                   \\
			=    & ~ TC_3' \int_0^{\tau_{k_s}} R(\tau) \Diff \tau \leq TC_3' \int_0^t R(\tau) \Diff \tau := C_3 \int_0^t R(\tau) \Diff \tau.
		\end{aligned}
	\end{equation}
	Similarly one can show that the forth term of \cref{eq:seventerm} can be bounded by
	\begin{equation}
		\Expect \left\|\sum_{k=0}^{k_s-1}\left[ \int_{\tau_k}^{\tau_{k+1}}\left(\frac{n-1}{2}\sigma_0^2X(\tau)\right)\Diff \tau -\delta_k\left(\frac{n-1}{2}\sigma_0^2Y_{k_\tau}\right)\right]\right\|_F^2	\leq C_4 \int_0^t R(\tau) \Diff \tau.
	\end{equation}

	The fifth term of \cref{eq:seventerm} can be estimated using Ito's isometry and the smoothness of $P_{uv}$:
	\begin{equation}
		\begin{aligned}
			     & ~\Expect \left\|\sigma_0\sum_{k=0}^{k_s-1}\left[\int_{\tau_k}^{\tau_{k+1}}P_{X(\tau)}(\Diff B(\tau)) - P_{Y_k}(\delta B_k)\right] \right\|_F^2              \\
			=    & ~\sigma_0^2 \Expect \left\| \sum_{u,v} \int_0^{\tau_{k_s}} [P_{uv}(X(\tau)-P_{uv}(Y_{k_\tau})] \Diff B_{uv}(\tau) \right\|_F^2                              \\
			=    & ~\sigma_0^2 \sum_{u,v} \int_0^{\tau_{k_s}} \Expect \|P_{uv}(X_{\tau}) - P_{uv}(Y_{k_\tau}) \|_F^2 \Diff \tau                                                \\
			\leq & ~ C_5 \int_0^{\tau_{k_s}} \Expect \|X(\tau) - Y_{k_\tau}\|_F^2 \Diff \tau = C_5 \int_0^{\tau_{k_s}} R(\tau) \Diff \tau \leq C_5 \int_0^t R(\tau) \Diff \tau.
		\end{aligned}
	\end{equation}

	The last two terms can be estimated as
	\begin{equation}
		\Expect \left\|\int_{\tau_{k_s}}^{s}\left(-\nabla_{\mathcal{M}}\Fun(X(\tau))-\frac{n-1}{2}\sigma_0^2X(\tau)\right) \Diff \tau \right\|_F^2
		\leq C_6 \delta^2.
		\label{eq:term:6}
	\end{equation}
	and
	\begin{equation}
		\Expect \left\|\sigma_0\int_{\tau_{k_s}}^{s}P_{X(\tau)}(\Diff B(\tau))\right\|_F^2 \leq C_7 \delta.
		\label{eq:term:7}
	\end{equation}

	Taking the above estimation together yields
	\begin{equation}
		R(t) \leq 7\left[(C_1+C_2+C_7) \delta + C_6 \delta^2 + (C_3+C_4+C_5) \int_0^t R(\tau) \Diff \tau \right].
	\end{equation}

	It follows directly from the Gronwall inequality that
	\begin{equation}
		R(t) \leq C\delta,  \quad \delta\in(0,\delta_0),
	\end{equation}
	where $C$ is A constant independent of $\delta$ and $\delta_0 > 0$.
\end{proof}

\section{IDDM Algorithm and Global convergence Analysis}
\label{sec:IDDM}

Now we can generalize existing methods based on diffusion equation \cref{eq:sde:stfmfd} to the Stiefel Manifold. The method we use in the following is a generalization of Intermittent Diffusion (ID) \cite{CHOW:2013jm}, namely Intermittent Diminishing Diffusion on Stiefel Manifold (IDDM), in which the diffusion strength is diminishing in every single cycle.


\begin{algorithm}
	\caption{Intermittent Diminishing Diffusion on Stiefel Manifold (IDDM)}
	\label{alg:iddm}
	\begin{algorithmic}[1]
		\REQUIRE {Maximum number of cycles $N$, diffusion strength $\sigma_n$, diffusion time (in one cycle) $T_n$, initial point $X_0$ (usually random selected)};
		\STATE $X_{\text{opt}} \leftarrow X_0, k \leftarrow 0$;
		\WHILE {Terminal conditions not satisfied}
		\IF {$k \geq N$}
		\STATE {\textbf{break}}
		\ENDIF
		\STATE Numerically solve \cref{eq:sde:ext:stra} by \cref{alg:ns} starting from $X_k$ using time $T_k$ and diffusion strength $\sigma(t)=\sigma_k$ to obtain $X'_{k+1}$;\\
		\STATE Solve $\Diff X_t=-\nabla_\Mfd \Fun(X_t) \Diff t$ by local
        algorithm starting from $X'_{k+1}$ untill convergence and get $X_{k+1}$;
		\IF {$f(X_{k+1})<f(X_{\text{opt}})$}
		\STATE $X_{\text{opt}} \leftarrow X_{k+1}$;
		\ENDIF
		\STATE $k \leftarrow k+1$;
		\ENDWHILE
	\end{algorithmic}
\end{algorithm}

We first notice that the method can also be viewed as selecting 
\begin{equation}
  \label{eq:sigma-up}\sigma(t) = \sum^N_{i=1}\sigma_i I_{[S_i,S_i+T_i]}(t),
\end{equation}
where $S_i$ is the starting time of each piece.

To provide the convergence results, we will first give some analysis for the
Fokker-Planck Equation \cref{eq:fpe}. The classic results yield the next
theorem. 
\begin{theorem}
	\label{thm:gibbs}
	Assume that $\sigma(t)=\sigma_0$ is a constant. The distribution $p(x,t)$ converges($\ell_1$) to the Gibbs distribution
	\begin{equation}
		\tilde{p}_{\sigma_0}(x)=\frac{1}{Z}e^{-2\Fun/\sigma^2_0},
		\label{eq:gibbs}
	\end{equation}
	where $Z$ is the normalization constant
	$\displaystyle Z = \int_\Mfd e^{-2\Fun/\sigma^2_0} $.
\end{theorem}
\begin{proof}
	To simplify the problem, we set $\sigma_0=\sqrt{2}$ or let $t'\leftarrow 2t/\sigma^2_0$ and $\Fun'\leftarrow 2\Fun/\sigma^2_0$ to transfer the Fokker-Planck equation to one with $\sigma_0=\sqrt{2}$.  Define the relative entropy
	\begin{equation}
		H(p\vert q)=\int_\Mfd p\log(\frac{p}{q})\Diff x,
		\label{eq:relent}
	\end{equation}
	for any two probability density function $p,q$ (on the manifold). The Csisz\'ar-Kullback inequality shows that
	\begin{equation}
		\|p-q\|^2_{\ell_1}\leq2H(p\vert q).
	\end{equation}
	The canonical Stiefel manifold is knows as Einstein manifold in the case of $n = p$ and thus the Ricci curvature is positive definite \cite{Oprea:2007vz}.

	For general canonical stiefel manifolds,  the same results are shown in \cite{Henkel:2005hk}.  It is given in	\cite{Markowich:2000up,Anonymous:EQrJJzRd} that $\Fun_0(X)=0$ satisfies a logarithmic Sobolev inequality with constant $\lambda_0$ which is the smallest eigenvalue of the Ricci curvature.
 It follows from \cite{Markowich:2000up,Anonymous:EQrJJzRd} that $\Fun(X)$ also satisfies a logarithmic Sobolev inequality with constant $\lambda=\lambda_0(\max_\Fun-\min_\Fun)$, which indicates that
	\begin{equation}
		H(p(:,t)\vert \tilde{p}_{\sigma_0})\leq e^{-2\lambda t}H(p(:,0)\vert \tilde{p}_{\sigma_0})).
	\end{equation}
	From the above analysis, we have
	\begin{equation}
		\|p(:,t)-\tilde{p}_{\sigma_0}\|^2_{\ell_1}\leq 2e^{-2\lambda t}H(p(:,0)\vert \tilde{p}_{\sigma_0}),
	\end{equation}
	which completes the proof.
\end{proof}
We next provide the convergence results of \cref{alg:iddm}, which is nearly the
same to the proof in \cite{CHOW:2013jm} in the Euclidean space.
\begin{theorem}[Convergence of \cref{alg:iddm}]
	\label{thm:id}
	Assume that the local algorithm satisfies $\Fun(X_k)\leq \Fun(X'_k)$.
	Let the set of global minimizers be $P$, the global minimum be $\Fun^*$, and
    $X_{opt}$ to be the optimal solution obtained by \cref{alg:iddm}. For any
    given $\epsilon>0$ and $\zeta>0$, let $U$ be the basin of global minima,
    i.e., $U = \{X \in \Mfd_{n,p}~|~ \Fun(X)<\Fun^*+\zeta\}$. Then the following two statements hold:
\begin{enumerate}
\item $\forall \eta\in(0,1)$, $\exists\sigma>0$ and $T>0$ (as a function of $\sigma$) such that if $\sigma_i\leq\sigma$ and $T_i>T(\sigma_i)$, then $\Prob(X'_i\in U) \geq \eta$.
\item The probability to reach $U$ after $N$ cycles is at least  $1 - (1 - \eta)^N$, namely, $\Prob(\exists i,\ X'_i\in U) > 1 - (1 - \eta)^N$. Thus, there exists $N_0>0$ such that if $\sigma_i\leq\sigma$, $T_i>T(\sigma_i)$ and $N>N_0$,
	\begin{equation}
		\Prob (\Fun(X_{opt})<\Fun^*+\zeta)\geq 1-\epsilon.
	\end{equation}
\end{enumerate}
\end{theorem}
\begin{proof}
	A small neighborhood $U$ can be given so that $\forall X\in U$, $\Fun(X)<\Fun^*+\zeta$. We only need to prove that $\Prob(\exists k,\ \mathrm{s.t.}\ X'_k\in U)\geq 1-\epsilon$.

	From \cref{eq:gibbs}, $\forall\eta\in (0,1)$, $\exists \sigma>0$ such that if $\sigma_i\leq\sigma$,
	\begin{equation}
		\int_U\tilde{p}_{\sigma_i}(x)\Diff x>\eta+(1-\eta)/2.
	\end{equation}
	Meanwhile, \cref{thm:gibbs} yields that $\exists T>0$ such that if $T_i>T$,
	\begin{equation}
		\|p(:,S_i+T_i)-\tilde{p}_{\sigma_i}\|_{\ell_1}<(1-\eta)/2.
	\end{equation}
	Hence, we have
	\begin{equation}
		\Prob(X'_i\in U)=\int_U \tilde{p}_{\sigma_i}dx-\int_U \tilde{p}_{\sigma_i}-p(x,S_i+T_i)dx\geq \eta .
	\end{equation}
	Independent intervals yields that
	\begin{equation}
		\Prob(\forall i,\ X'_i\in U^c)<(1-\eta)^N.
	\end{equation}
	Select a proper $N_0$ such that $(1-\eta)^{N_0}\leq\epsilon$ and we complete the proof.
\end{proof}

\begin{remark} {\rm The results of \cref{thm:id} can be improved if we impose
  some stronger conditions on the object function and the local algorithm. If 1)
  the local algorithm always achieve a nearest local minimizer and 2) there are
  finite local minimizers (which is acceptable for a compact set), then the
  results can be improved as $\Prob(\mathrm{dist}(X_{opt},P)<\zeta)\geq
  1-\epsilon$. The proof is the same as \cite{CHOW:2013jm}.} \end{remark}

\begin{remark} {\rm 
 We have provided some analysis for the piecewise constant $\sigma(t)$
proposed by \cite{CHOW:2013jm}. Notice that other $\sigma(t)$ may also give
global convergence. For example, one can apply the
$\sigma(t)=c/\sqrt{\log(t+2)}$ given by CDD and the proof of convergence is the
same. One can refer to \cite{Chiang:1987ip,Geman:1986js} for the proof.}
\end{remark}
\section{Numerical Experiments}
\label{sec:experiments}
\graphicspath{ {./figures/} }

In this section, we demonstrate the effectiveness of IDD methods on Stiefel
Manifold (IDDM) on a variety of test problems. The first two subsections are
devoted to the spherically constrained problems, while the last one focuses on
the orthogonality constrained problem. We should point out that we have also test
many problems, such as conformal mapping~\cite{gu2003global,lai2014folding}, p-Harmonic
flow~\cite{lin1989relaxation,tang2001color,vese2002numerical,goldfarb2009curvilinear},
 compressed
modes~\cite{ozolicnvs2013compressed} and nonlinear eigenvalue problem in
density functional theory \cite{scf-sinum,TRDFT}. They are not choosen in this
section because the local algorithm is often able to return a pretty solution
(or even ``global solution'') in a single run.

The performance of IDDM is mainly compared with the \emph{Random-Start} local method dubbed as RSlocal, which randomly selects an initial point and then performs the local algorithm. The local algorithm that we employ is the curvilinear search method with Barzilai-Borwein steps (Algorithm 2 in \cite{Wen:2012ga}). 
Each run of IDDM consists of ten cycles  while RSlocal is made up
of ten trials of the local algorithm  starting from randomly
generated points. The parameter $\sigma_i$ in \cref{eq:sigma-up} is set to
$\sigma_i = \alpha/(i d_t)^{1/2(n-1)}$, where $d_t$ is the step length, $n$
is the dimension of the variables and $\alpha$ is the initial diffusion strength.
All experiments were performed on a workstation with an Intel Xeon E5-2640 v3 \@ 2.60GHz processor with access to 64 GB of RAM.
\subsection{Homogeneous Polynomial Optimization} 
In this subsection, we evaluate the performance on homogeneous polynomial
problems. The test polynomial is selected from \cite{Wen:2012ga} which cannot be
globally minimized effectively by the local methods:
\begin{equation} \label{prob:hp1}
\min_{x\in \mathbb{R}^n}\quad \Fun(x) = \sum_{1 \leq i \leq n} x_i^6 + \sum_{1 \leq i \leq n-1}x_i^3x_{i+1}^3, \quad 
\mathrm{s.t.}\quad \|x\|_2 = 1.\end{equation}
For each of $n=10,20,\ldots,200$, we repeat $50$ independent runs of IDDM and
RSlocal. The initial diffusion strength $\alpha$ is
selected as $1/n$. The minimum, mean and maximum of the objective function
values, as well as the averaged cpu time in seconds are reported in
\cref{tab:eg1}.  The corresponding mean and min are further illustrated in  the left side of \cref{fig:reseg1}. 
Our numerical results indicate that IDDM are always much better than RSlocal in this problem.

We further numerically explore the dependency of the performance of IDDM  to the
diffusion strength $\alpha$.  For each $n$ ranging from $40$ to $200$, we repeat $50$ independent tests of RSlocal and denote the  averaged objective function values as $\Fun_{\mathrm{RSlocal}}$.  Similarly, we repeat $50$ runs of IDDM with different initial diffusion strengths $\sigma$ from $10^{-4}$ to $10^{0}$ and the averaged objective function values are denoted by $\Fun_{\mathrm{IDDM}}$.  Each pixel in the right side of \cref{fig:reseg1} is a value of $-\log_{10} (\Fun_{\mathrm{IDDM}} /\Fun_{\mathrm{RSlocal}})$.  A positive value indicates an improvement achieved by IDDM over RSlocal while a negative value means that IDDM is worse than RSlocal. 
This image clearly shows that our IDDM outperforms RSlocal with the right choice of the diffusion strength illustrated in the region with the red color. 

\begin{table}[hbp]
\centering
\footnotesize
\setlength{\tabcolsep}{5pt}
\label{tab:eg1}
\caption{Numerical results of polynomial optimization \cref{prob:hp1}}
\begin{tabular}{|c|ccc|c|ccc|c|}
\hline
\multirow{2}{*}{$n$} & \multicolumn{4}{c|}{RSlocal} & \multicolumn{4}{c|}{IDDM}  \\ \cline{2-9} 
    &   min     &    mean & 	max       &   cpu (s)     &    min      &  mean         &    max     &	cpu (s)           \\ \hline
10 & 7.2e-04 & 2.5e-03 & 6.3e-03 & 0.041 & 7.2e-04 & 2.5e-03 & 4.8e-03  & 0.044  \\ \hline 
20 & 1.1e-04 & 6.1e-04 & 1.3e-03 & 0.047 & 5.0e-05 & 3.3e-04 & 1.1e-03  & 0.087  \\ \hline 
30 & 7.6e-05 & 2.8e-04 & 4.8e-04 & 0.057 & 1.5e-05 & 1.1e-04 & 2.5e-04  & 0.116  \\ \hline 
40 & 7.5e-05 & 1.6e-04 & 2.5e-04 & 0.065 & 5.4e-06 & 4.5e-05 & 1.2e-04  & 0.129  \\ \hline 
50 & 5.4e-05 & 1.1e-04 & 1.7e-04 & 0.078 & 2.2e-06 & 2.1e-05 & 5.5e-05  & 0.166  \\ \hline 
60 & 2.7e-05 & 8.0e-05 & 1.4e-04 & 0.087 & 2.1e-06 & 1.3e-05 & 3.5e-05  & 0.154  \\ \hline 
70 & 2.9e-05 & 6.0e-05 & 9.0e-05 & 0.100 & 1.0e-06 & 9.9e-06 & 2.6e-05  & 0.173  \\ \hline 
80 & 2.0e-05 & 4.5e-05 & 6.3e-05 & 0.095 & 1.3e-06 & 8.1e-06 & 4.2e-05  & 0.153  \\ \hline 
90 & 2.3e-05 & 3.7e-05 & 5.3e-05 & 0.099 & 1.2e-06 & 5.4e-06 & 1.7e-05  & 0.169  \\ \hline 
100 & 1.7e-05 & 3.1e-05 & 4.5e-05 & 0.104 & 6.7e-07 & 4.3e-06 & 1.4e-05  & 0.175  \\ \hline 
110 & 1.6e-05 & 2.6e-05 & 3.6e-05 & 0.118 & 4.6e-07 & 3.0e-06 & 8.6e-06  & 0.166  \\ \hline 
120 & 1.5e-05 & 2.3e-05 & 3.1e-05 & 0.115 & 6.2e-07 & 2.6e-06 & 8.9e-06  & 0.174  \\ \hline 
130 & 8.3e-06 & 1.9e-05 & 2.5e-05 & 0.125 & 3.2e-07 & 2.3e-06 & 4.6e-06  & 0.191  \\ \hline 
140 & 6.8e-06 & 1.6e-05 & 2.1e-05 & 0.128 & 7.5e-07 & 2.4e-06 & 7.3e-06  & 0.172  \\ \hline 
150 & 9.9e-06 & 1.5e-05 & 2.1e-05 & 0.136 & 6.7e-07 & 2.0e-06 & 6.9e-06  & 0.180  \\ \hline 
160 & 8.3e-06 & 1.3e-05 & 1.8e-05 & 0.140 & 3.4e-07 & 1.8e-06 & 7.5e-06  & 0.184  \\ \hline 
170 & 7.9e-06 & 1.1e-05 & 1.4e-05 & 0.153 & 3.2e-07 & 1.9e-06 & 5.0e-06  & 0.183  \\ \hline 
180 & 7.9e-06 & 1.0e-05 & 1.3e-05 & 0.152 & 5.4e-07 & 1.8e-06 & 5.0e-06  & 0.186  \\ \hline 
190 & 7.3e-06 & 9.5e-06 & 1.2e-05 & 0.154 & 4.1e-07 & 1.7e-06 & 3.4e-06  & 0.190  \\ \hline 
200 & 5.2e-06 & 8.5e-06 & 1.1e-05 & 0.160 & 5.8e-07 & 2.0e-06 & 7.0e-06  & 0.189  \\ \hline 
\end{tabular}
\end{table}

\begin{figure}[!hbt]
\centering
\label{fig:reseg1}
\caption{(a) The objective function values of IDDM an RSlocal on
\cref{prob:hp1}. (b) $-\log_{10} (\Fun_{\mathrm{IDDM}} /\Fun_{\mathrm{RSlocal}})$, i.e.,  the performance of IDDM using various initial diffusion strength with respect to \emph{RSlocal}.}
    \subfigure[]{
\includegraphics[width = 0.47\textwidth]{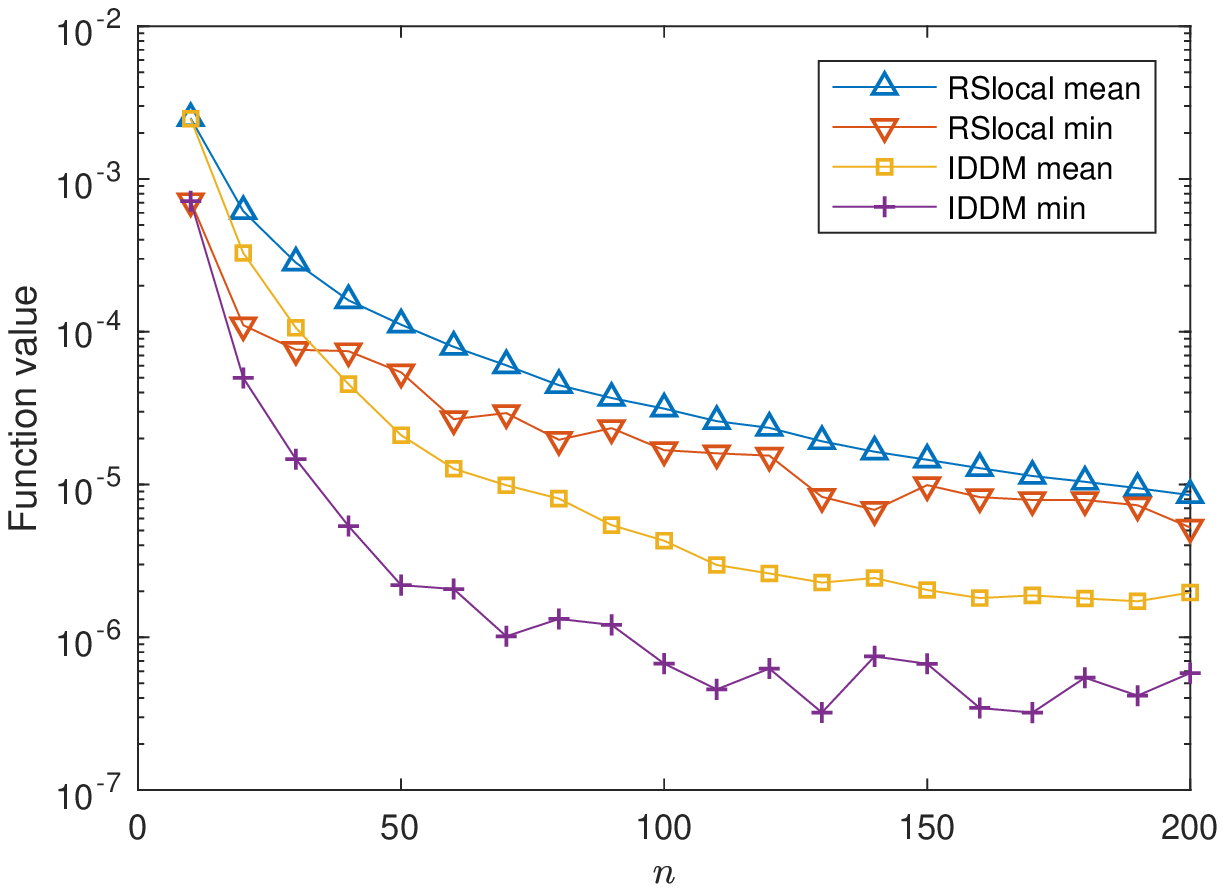}
  }
  \hfill
    \subfigure[]{
\includegraphics[width = 0.47\textwidth]{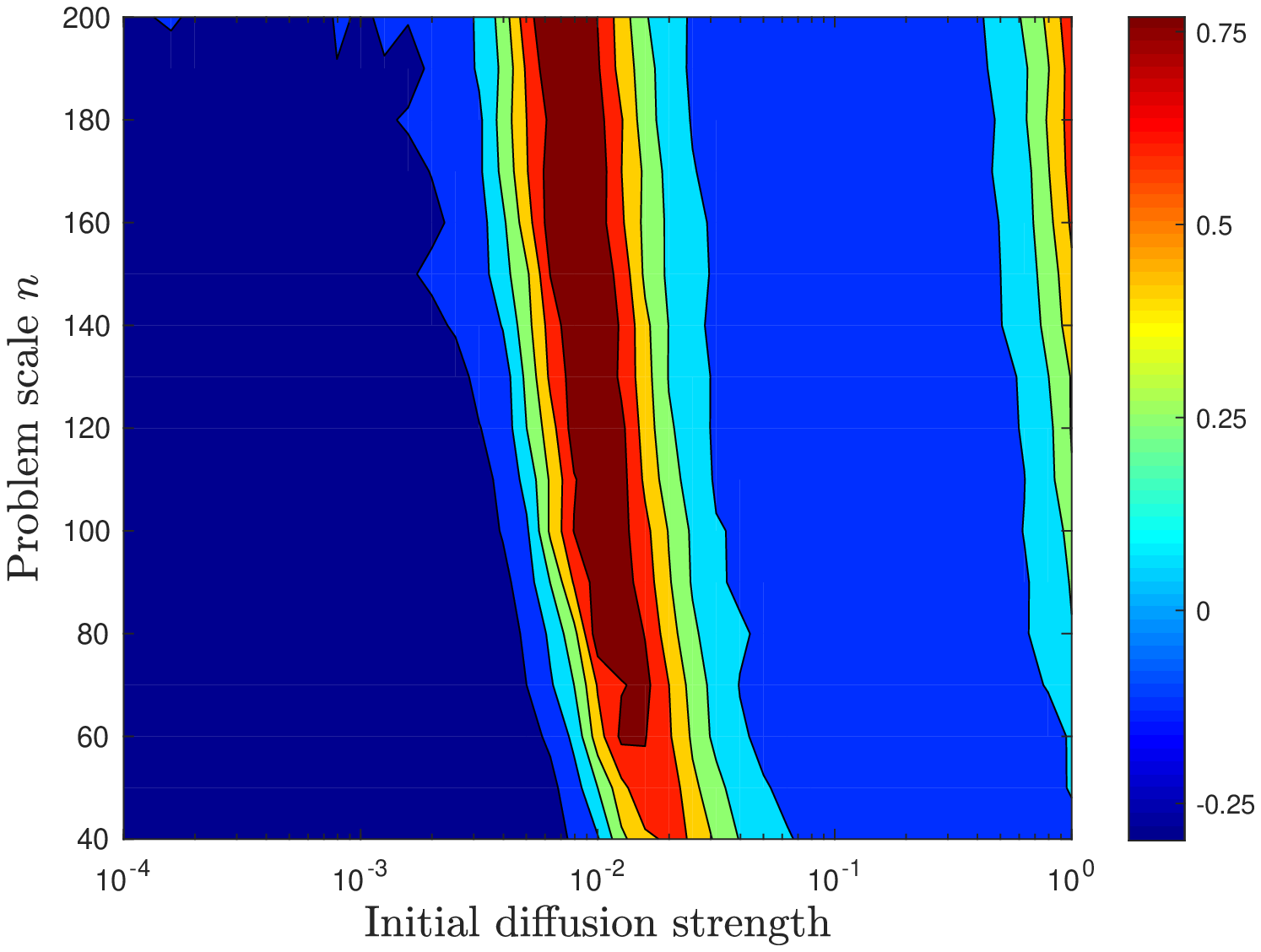}
  }

\end{figure}


\subsection{Biquadratic optimization}
We next consider the so-called biquadratic optimization over unit spheres
\cite{LinNieQi09}: 
\begin{equation}\label{prob:biquad}
\begin{aligned}
	\min_{x\in \mathbb{R}^n, y \in \mathbb{R}^n} \quad & b(x,y) = \sum_{1 \leq i, k \leq n, 1 \leq j,l \leq n} b_{ijkl} x_iy_jx_ky_l \\
	\text{s.t.} \quad & \|x\| = 1, \|y\| = 1.
\end{aligned}
\end{equation}
Without loss of generality, we impose the symmetric property $b_{ijkl} =
b_{kjil} = b_{ilkj}$ for $i, k, j,l = 1,\dots, n$.
A semidefinite programing relaxation approach is proposed in \cite{LinNieQi09}.
Since Examples 5.1 to 5.3 in this reference can be easily found by local solvers,
we generate the coefficients $b_{ijkl}$ as following:
\begin{enumerate}
  \item[case i)] $b_{ijkl} = (-1)^{i+j+k+l}|c|$, where $c$ is a Gaussian random
    variable.
  \item[case ii)] $b_{ijkl} = |c_1| 1_{c_2 > \eta}$, where $c_1$ is a Gaussian random
    variable, $c_2$ is uniformly distributed in $[0,1]$ and $\eta \in (0,1)$.
 \end{enumerate}
For each of $n=6,7,\ldots,25$, we repeat $50$ independent runs of IDDM and
RSlocal. For the parameter $\alpha$ of IDDM,  we select a few values in
$[10^{-4},10^2]$ for each $n$ and choose the one with the best performance.    The minimum, mean and maximum of the difference between the objective function
values and the smallest objective function value identified in the $50$ runs
are reported in Tables \ref{biq1} and \ref{biq2}. 
 From the tables, we can see that both IDDM and RSlocal can find the
 ``smallest'' function values. 
 IDDM
 usually performs better than RSlocal in most cases in terms of the mean value.

\begin{table}[hbp] \label{biq1}
	\centering
    \setlength{\tabcolsep}{5pt}
\footnotesize
	\caption {Numerical results of biquadratic optimization: case i}
		\begin{tabular}{|c|cccc|cccc|}
		\hline 
\multirow{2}{*}{$n$} & \multicolumn{4}{c|}{RSlocal} & \multicolumn{4}{c|}{IDDM}
\\ \cline{2-9} 
    &   min     &    mean & 	max      &   cpu (s)      &    min      &  mean         &    max      &	cpu  (s)        \\ \hline
              
6 & 5.2e-14 & 1.4e-02 & 2.5e-01   & 0.033 & 1.5e-14 & 8.3e-03 & 1.8e-02  & 0.034  \\ \hline 
7 & 3.6e-14 & 2.2e-02 & 2.8e-01   & 0.020 & 3.7e-14 & 1.4e-02 & 2.8e-01  & 0.030  \\ \hline 
8 & 4.8e-14 & 1.1e-01 & 1.6e+00   & 0.026 & 2.7e-15 & 2.5e-13 & 7.1e-13  & 0.030  \\ \hline 
9 & 2.5e-14 & 1.4e-01 & 3.9e-01   & 0.032 & 2.2e-14 & 4.3e-12 & 2.6e-11  & 0.037  \\ \hline 
10 & 3.3e-14 & 1.1e-01 & 1.2e+00  & 0.041 & 7.1e-15 & 3.4e-12 & 1.9e-11  & 0.046  \\ \hline 
11 & 2.7e-14 & 8.0e-02 & 6.4e-01  & 0.046 & 2.8e-14 & 1.6e-11 & 1.8e-10  & 0.054  \\ \hline 
12 & 9.9e-14 & 3.5e-02 & 2.1e-01  & 0.056 & 1.4e-13 & 1.9e-02 & 1.6e-01  & 0.075  \\ \hline 
13 & 3.7e-14 & 2.5e-01 & 7.7e-01  & 0.059 & 2.0e-14 & 2.0e-01 & 7.7e-01  & 0.076  \\ \hline 
14 & 1.1e-13 & 2.4e-01 & 1.1e+00  & 0.073 & 8.2e-14 & 1.4e-01 & 9.0e-01  & 0.102  \\ \hline 
15 & 1.2e-13 & 1.3e-01 & 5.2e-01  & 0.085 & 1.2e-13 & 4.8e-12 & 6.4e-11  & 0.088  \\ \hline 
16 & 5.0e-13 & 5.2e-02 & 4.6e-01  & 0.110 & 2.4e-12 & 2.7e-02 & 3.1e-01  & 0.163  \\ \hline 
17 & 1.4e-14 & 6.5e-01 & 1.5e+00  & 0.112 & 3.0e-13 & 3.9e-12 & 6.2e-11  & 0.111  \\ \hline 
18 & 2.0e-13 & 1.9e-01 & 9.5e-01  & 0.136 & 6.6e-14 & 6.1e-02 & 4.6e-01  & 0.187 \\ \hline
19 & 2.9e-13 & 3.6e-01 & 9.5e-01  & 0.186 & 1.4e-14 & 2.5e-12 & 1.9e-11  & 0.178  \\ \hline 
20 & 4.1e-14 & 4.0e-01 & 1.3e+00  & 0.225 & 2.8e-13 & 2.3e-01 & 9.8e-01  & 0.335  \\ \hline 
21 & 4.3e-14 & 4.9e-01 & 1.2e+00  & 0.267 & 4.4e-14 & 3.0e-01 & 8.7e-01  & 0.409  \\ \hline 
22 & 3.3e-13 & 4.2e-01 & 1.0e+00  & 0.324 & 4.4e-13 & 2.9e-11 & 1.4e-10  & 0.472  \\ \hline 
23 & 1.0e-13 & 6.9e-01 & 1.8e+00  & 0.410 & 2.7e-13 & 6.7e-12 & 2.5e-11  & 0.454  \\ \hline 
24 & 1.2e-13 & 4.7e-01 & 1.1e+00  & 0.484 & 1.2e-12 & 3.4e-01 & 9.9e-01  & 0.711  \\ \hline 
25 & 5.6e-13 & 3.4e-01 & 1.1e+00  & 0.556 & 5.1e-13 & 3.1e-01 & 1.2e+00  & 0.876  \\ \hline 
	\end{tabular}	
\end{table}
 
 \begin{table}[hbp] \label{biq2}
	\centering
    \setlength{\tabcolsep}{5pt}
    \footnotesize
	\caption {Numerical results of biquadratic optimization: case ii}
		\begin{tabular}{|c|cccc|cccc|}

		\hline 
\multirow{2}{*}{$n$} & \multicolumn{4}{c|}{RSlocal} & \multicolumn{4}{c|}{IDDM}
\\ \cline{2-9} 
    &   min     &    mean & 	max      &   cpu      &    min      &  mean         &    max     &	cpu           \\ \hline
    
 6 & 4.0e-14 & 3.6e-04 & 5.9e-03  & 0.014 & 4.7e-14 & 3.6e-04 & 5.9e-03  & 0.020  \\ \hline 
7 & 3.2e-14 & 7.1e-03 & 1.2e-01   & 0.016 & 4.2e-14 & 1.1e-13 & 2.8e-13  & 0.016  \\ \hline 
8 & 0.0e+00 & 8.3e-03 & 8.8e-02   & 0.032 & 4.9e-15 & 2.4e-03 & 4.0e-02  & 0.042  \\ \hline 
9 & 5.6e-14 & 7.0e-02 & 4.2e-01   & 0.020 & 2.3e-14 & 2.2e-02 & 4.2e-01  & 0.028  \\ \hline 
10 & 2.0e-14 & 6.7e-02 & 2.4e-01  & 0.029 & 1.4e-14 & 3.5e-02 & 1.8e-01  & 0.040  \\ \hline 
11 & 9.5e-14 & 7.4e-03 & 7.1e-02  & 0.031 & 6.8e-14 & 2.4e-03 & 6.3e-03  & 0.046  \\ \hline 
12 & 1.2e-13 & 4.6e-02 & 2.3e-01  & 0.036 & 4.4e-14 & 2.0e-12 & 2.0e-11  & 0.045  \\ \hline 
13 & 1.4e-13 & 5.2e-02 & 2.7e-01  & 0.039 & 4.3e-14 & 2.5e-02 & 2.9e-01  & 0.061  \\ \hline 
14 & 5.7e-14 & 2.9e-01 & 8.2e-01  & 0.056 & 3.5e-14 & 1.0e-01 & 7.1e-01  & 0.084  \\ \hline 
15 & 7.1e-15 & 6.0e-02 & 2.7e-01  & 0.064 & 3.8e-14 & 2.7e-02 & 3.0e-01  & 0.092  \\ \hline 
16 & 3.6e-14 & 1.3e-01 & 9.2e-01  & 0.091 & 5.0e-14 & 2.0e-12 & 3.1e-12  & 0.125  \\ \hline 
17 & 3.5e-14 & 5.6e-02 & 7.8e-01  & 0.091 & 1.3e-14 & 5.1e-13 & 7.6e-12  & 0.141  \\ \hline 
18 & 2.3e-13 & 2.8e-01 & 8.0e-01  & 0.118 & 5.4e-14 & 8.9e-14 & 1.5e-13  & 0.165  \\ \hline 
19 & 1.4e-13 & 1.1e-01 & 4.5e-01  & 0.146 & 1.2e-14 & 9.3e-02 & 2.7e-01  & 0.232  \\ \hline 
20 & 6.8e-14 & 1.8e-01 & 7.0e-01  & 0.202 & 2.4e-13 & 1.2e-01 & 7.1e-01  & 0.282  \\ \hline 
21 & 1.5e-13 & 1.3e-01 & 3.7e-01  & 0.235 & 3.6e-14 & 9.4e-02 & 2.3e-01  & 0.366  \\ \hline 
22 & 1.8e-13 & 2.3e-01 & 5.5e-01  & 0.291 & 7.0e-13 & 1.4e-01 & 5.1e-01  & 0.469  \\ \hline 
23 & 1.0e-12 & 1.3e-01 & 4.8e-01  & 0.360 & 5.6e-13 & 1.1e-01 & 5.1e-01  & 0.550  \\ \hline 
24 & 1.1e-12 & 2.1e-01 & 4.8e-01  & 0.389 & 9.9e-14 & 1.9e-02 & 4.2e-01  & 0.610  \\ \hline 
25 & 9.4e-13 & 9.6e-02 & 3.4e-01  & 0.507 & 0.0e+00 & 8.7e-02 & 3.3e-01  & 0.771  \\ \hline 
 	\end{tabular}	
\end{table}

 We next  demonstrate the performance of IDDM  with respect to the initial
 diffusion strength $\alpha$.  For each $n=\{18,20\}$, we repeat $50$ independent tests of RSlocal.  The  averaged difference to global objective function values is plotted as the red line in \cref{fig:biquad-diff}. Then we repeat $50$ runs of IDDM with different initial diffusion strengths $\sigma$ from $10^{-4}$ to $10^{2}$.  The averaged  difference to global objective function values are depicted as the blue curve in \cref{fig:biquad-diff}.  
 We can see that our IDDM outperforms RSlocal if the diffusion strength is chosen suitably. Similar behavior can be observed on other dimensions of $n$.
 \begin{figure}[hbp]
\centering
\label{fig:biquad-diff}
    \subfigure[$n=18$]{
\includegraphics[width = 0.47\textwidth]{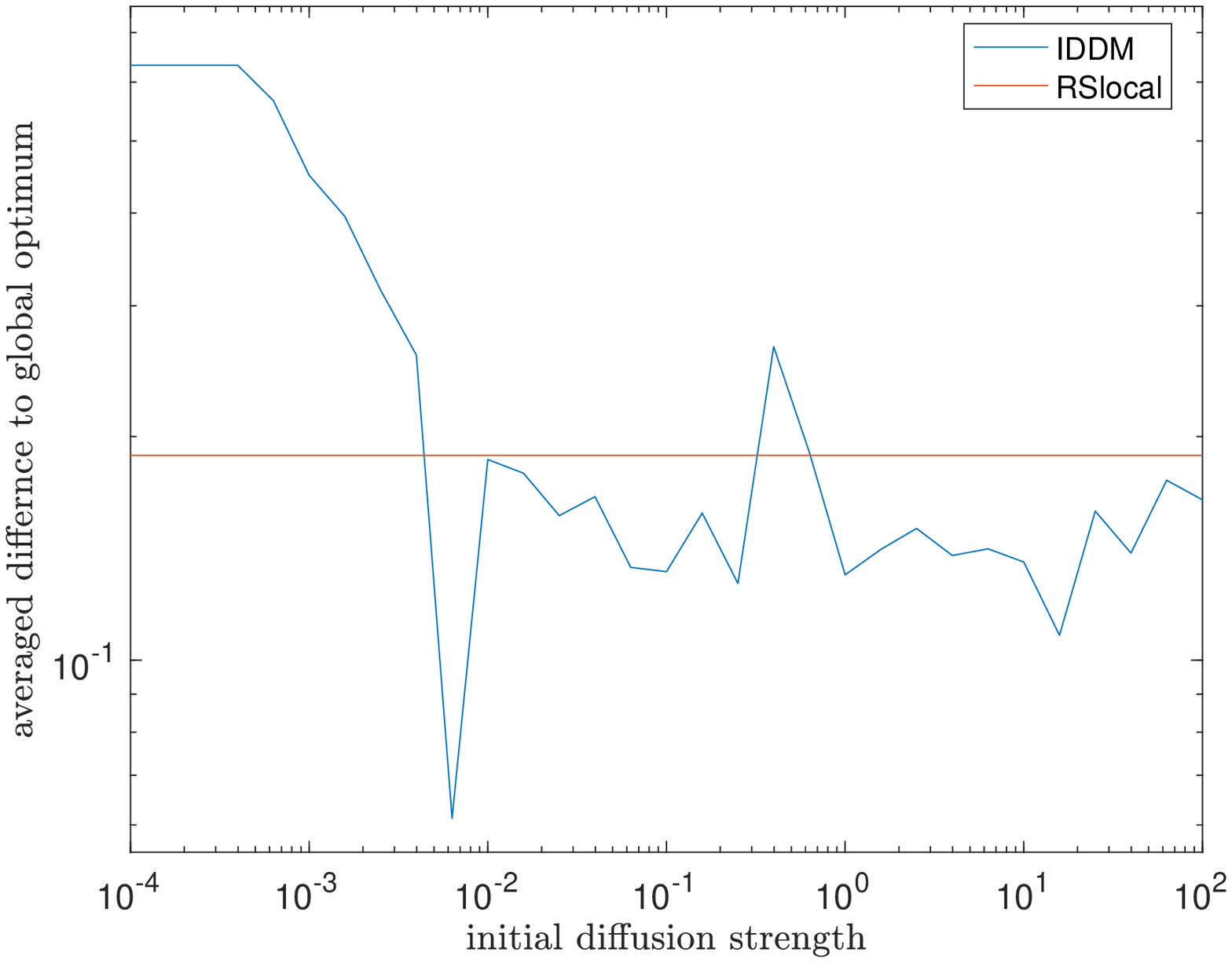}
  }
  \hfill
    \subfigure[$n=20$]{
\includegraphics[width = 0.47\textwidth]{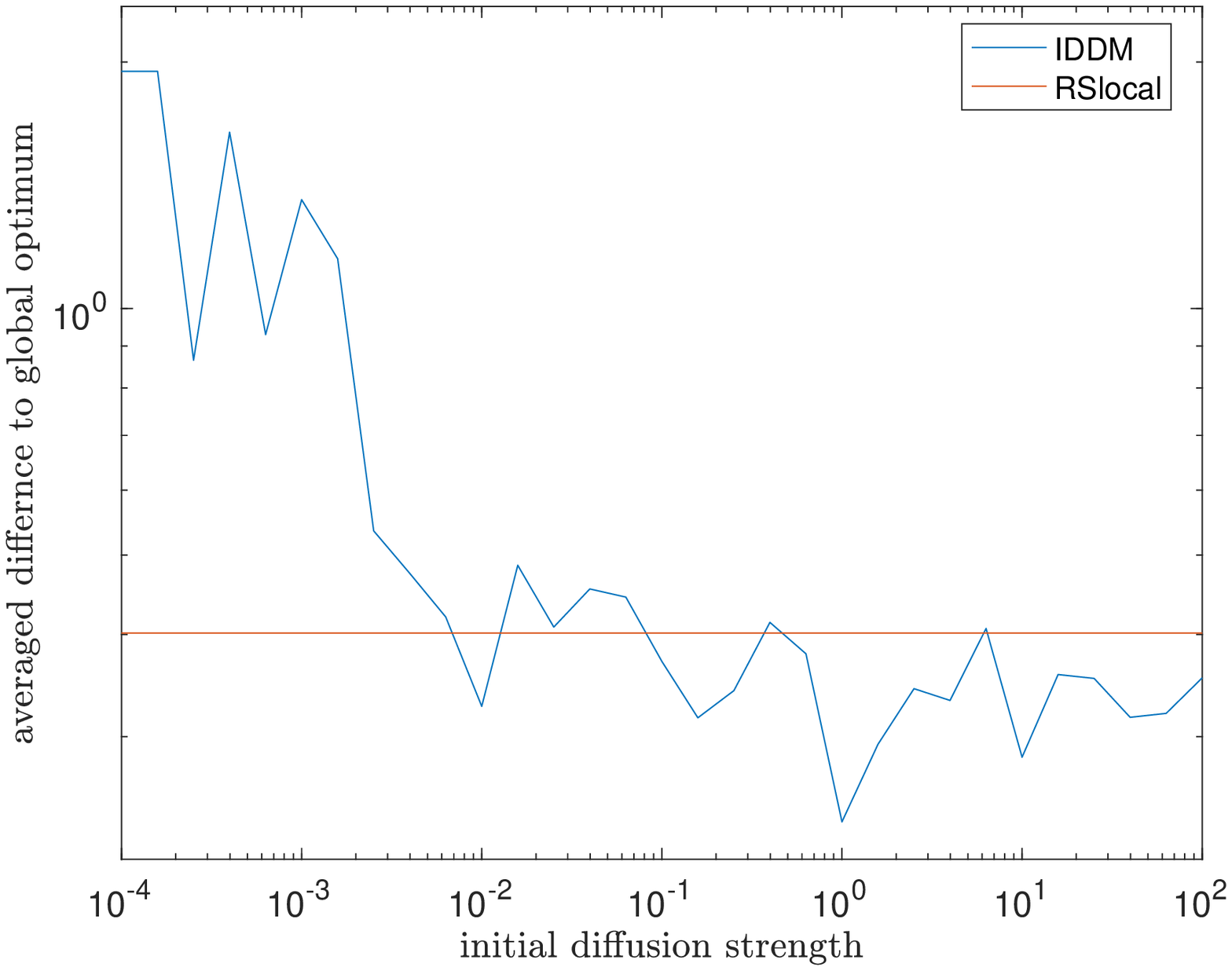}
  }
  \caption{The performance of IDDM with respect to the initial diffusion strength}
\end{figure}

\subsection{Computation of Stability Number}
Let $G = (V,E)$ be an undirected graph. A stable (independent) set in $G$ is a
set of vertices that are mutually nonadjacent. The \emph{stability number} $S(G)$ for a given graph $G$ is defined as the size of a maximum stable set in $G$. 
 It was shown by Motzkin and Straus \cite{Motzkin:1965fn} that 
\[S(G)^{-1} = \min_{\|x\|_2 = 1} \sum_{i=1}^n x_i^4 + 2\sum_{(i,j) \in E} x_i^2 x_j^2,\]
which is a single spherically constrainted problem. We select a few typical
graphs as in \cite{Wen:2012ga} and we repeat $50$ independent runs of IDDM and
RSlocal. The parameter $\alpha$ is set to $0.005$ in IDDM. The size $|V|$ of the graph, the  mean and maximum of $S(G)$ as well as the cpu time are presented in
\cref{tab:stab}.  Note that the larger the value $S(G)$ is obtained, the better the
stability number is estimated. We can see that IDDM almost always achieve a
better solution than RSlocal. 
\begin{table}[htbp]
  \setlength{\tabcolsep}{4pt}
  \footnotesize
\centering
\label{tab:stab}
\caption{Stability Number}
\begin{adjustbox}{max width=\textwidth}
\begin{tabular}{|cc|ccc|ccc|ccc|}
\hline
\multicolumn{2}{|c|}{graph} & \multicolumn{3}{c|}{RSlocal}                      & \multicolumn{3}{c|}{IDDM}                         \\ \hline
name              & $|V|$       & mean  & max  & cpu (s) & mean & max  & cpu (s)\\ \hline
        theta10 &    500 &47.0 & 50 &    0.686 & 47.0 & 51 &    0.620 \\ \hline 
        theta12 &    600 &49 & 50 &    0.949 & 49 & 54 &    0.872 \\ \hline 
        theta42 &    200 &15 & 17 &    0.254 & 15.5 & 18 &    0.265 \\ \hline 
      
            G43 &   1000 &180.5 & 188 &    0.671 & 189.0 & 195 &    0.497 \\ \hline 
            G44 &   1000 &182.0 & 190 &    0.654 & 191.0 & 199 &    0.503 \\ \hline 
            G45 &   1000 &179.0 & 188 &    0.667 & 187.5 & 197 &    0.495 \\ \hline 
            G46 &   1000 &180.0 & 186 &    0.651 & 189.0 & 196 &    0.505 \\ \hline 
            G47 &   1000 &184.0 & 190 &    0.689 & 191.5 & 200 &    0.487 \\ \hline 
            G51 &   1000 &332.0 & 336 &    0.813 & 343.0 & 346 &    0.603 \\ \hline 
            G52 &   1000 &330.0 & 335 &    0.837 & 341.0 & 344 &    0.616 \\ \hline 
            G53 &   1000 &330.0 & 334 &    0.783 & 340.0 & 343 &    0.557 \\ \hline 
            G54 &   1000 &323.0 & 330 &    0.725 & 334.0 & 339 &    0.532 \\ \hline 
    sanr200-0.7 &    200 &16.0 & 17 &    0.262 & 15.0 & 18 &    0.275 \\ \hline 
     brock200-4 &    200 &14.0 & 15 &    0.273 & 14.0 & 17 &    0.275 \\ \hline 
     hamming-6-4 &     64 &4.0 & 4 &    0.033 & 4.0 & 4 &    0.032 \\ \hline 
    hamming-9-8 &    512 &168.0 & 179 &    0.264 & 173.0 & 186 &    0.089 \\ \hline 
   hamming-10-2 &   1024 &65.0 & 67 &    0.911 & 66.0 & 70 &    0.844 \\ \hline 
   hamming-11-2 &   2048 &113.0 & 116 &    2.267 & 118.0 & 122 &    1.889 \\ \hline 
        keller4 &    171 &9.0 & 11 &    0.247 & 11.0 & 11 &    0.172 \\ \hline 
          fap25 &   2118 &78.0 & 80 &   29.471 & 79.0 & 82 &   25.063 \\ \hline 
       1dc.1024 &   1024 &69.0 & 71 &    1.092 & 70.0 & 73 &    1.006 \\ \hline 
       1dc.2048 &   2048 &119.0 & 123 &    2.679 & 125.0 & 129 &    2.290 \\ \hline 
        1et.512 &    512 &91.0 & 96 &    0.202 & 92.0 & 96 &    0.175 \\ \hline 
       1et.1024 &   1024 &154.0 & 158 &    0.426 & 159.0 & 162 &    0.363 \\ \hline 
       1et.2048 &   2048 &270.0 & 275 &    0.971 & 289.0 & 296 &    0.894 \\ \hline 
        1tc.512 &    512 &101.0 & 104 &    0.165 & 103.0 & 106 &    0.156 \\ \hline 
       1tc.1024 &   1024 &174.0 & 180 &    0.357 & 183.0 & 187 &    0.302 \\ \hline 
       1tc.2048 &   2048 &305.0 & 312 &    0.804 & 323.5 & 329 &    0.740 \\ \hline 
       1zc.512 &    512 &51.0 & 54 &    0.299 & 51.5 & 55 &    0.286 \\ \hline 
       1zc.1024 &   1024 &91.0 & 95 &    0.728 & 93.0 & 99 &    0.648 \\ \hline 
       1zc.2048 &   2048 &160.0 & 164 &    1.740 & 169.0 & 175 &    1.388 \\ \hline 
       1zc.4096 &   4096 &286.0 & 292 &    4.502 & 289.0 & 296 &    4.673 \\ \hline 
        2dc.512 &    512 &10.0 & 10 &    2.303 & 10.0 & 11 &    2.001 \\ \hline 
        
\end{tabular}
\end{adjustbox}
\end{table}

\subsection{Structure Determination in Cryo-EM}

We now consider an example with multiple orthogonality constraints that arises
from Cryo-EM \cite{Singer:2011ba}. In this test
problem, we try to recover $N$ orientations $\{\tilde{R}_i\}$ from two
dimensional (2D) projection images $\{P_i\}$ of a three dimensional (3D) object.
Each $\tilde{R}_i\in\mathbb{R}^{3\times 3}$ describes a 3D orthogonal matrix or
rotation, i.e.,  $\tilde{R}^\top_i\tilde{R}_i = I_3$ and $\det(\tilde{R}_i)=1$. 
Let $\tilde{c}_{ij} = (x_{ij},y_{ij},0)$ be the common line of the Fourier
transforms of $P_i$ and $P_j$
(viewed in $P_i$).
When the data are exact, it follows from the Fourier
projection-slice theorem \cite{Singer:2011ba} that  the common lines must coincide, i.e.,
\[\tilde{R}_i\tilde{c}_{ij}=\tilde{R}_j\tilde{c}_{ji}.\]
Since the third column  $\tilde{R}^3_i$ can be recovered from the first two
columns $\tilde{R}^1_i$ and $\tilde{R}^2_i$ as
$\tilde{R}^3_i=\pm\tilde{R}^1_i\times\tilde{R}^2_i$, the rotations $\{\tilde{R}_i\}$ can be compressed to $3\times 2$ matrix.  
Therefore, the corresponding optimization problem can be formulated as
\begin{equation}
	\begin{aligned}
		\min_{R_i}  \quad \sum^N_{i=1}\rho(R_i c_{ij},R_j c_{ji}),\quad                     
		\mathrm{s.t.}  \quad R^\top_iR_i=I_2,R_i\in\Real^{3\times 2}
	\end{aligned}
    \label{eq:cryoem}
\end{equation}
where $ \rho$ is the function representing the distance between the two vectors,
$R_i$ is made up of the first two columns of $\tilde{R}_i$ and $c_{ij}$ consists
of the first two elements of $\tilde{c}_{ij}$.  The distance function $\rho(u,v)=\|u-v\|_2$ is chosen in
\cite{Singer:2011ba} and it leads to an eigenvector relaxation and semidfinite
programming relxation. In our experiments, we select
$\rho(u,v)=\|u-v\|_q$ with $q=0.5$ since it often leads to better mean square error defined as follows. 
 Note that it holds $O\tilde{R}_i\tilde{c}_{ij}=O\tilde{R}_j\tilde{c}_{ji}$ for
 any 
fixed orthogonal matrix $ O \in \mathbb{R}^{3\times 3}$.
 Hence, we measure the error between the recovered rotations ${\hat{R}}_i$ and real rotations
$\tilde{R}_i$ by the mean square error (MSE) defined as
\[\mathrm{MSE}=\min_{O^\top O = I_3}\sum_{i=1}^N\|{\hat{R}}_i-O\tilde{R}_i\|^2_F.\]

We compare IDDM with RSlocal and the
eigenvector relaxation method  developed in
\cite{Singer:2011ba} (dubbed as ``eigs''). The semidefinite programming
relaxation approach in \cite{Singer:2011ba} is not compared because our
experiments show that our local algorithm often can be better than it in terms
of both accuracy and computational time.  Each run of IDDM consists of ten cycles starting from the
point generated from eigs  while RSlocal is made up
of ten trials of the local algorithm  starting either from eigs or nine randomly
generated points. The parameter $\alpha$ in IDDM is set to $0.1$. In the subsequent experiments, 
``cpu'' is the average cpu time of
one cycle in seconds and  ``obj'' stands for the final objective value.

Our first experiment is based on randomly generated data sets. We first create
$N$ rotations $\tilde{R}_i$ by using the MATLAB command ``orth(rand(3,3))''. The
common line vectors are computed next as  $\tilde{c}_{ij} =
\tilde{R}_i^{-1}\cdot (\tilde{R}^3_i\times
\tilde{R}^3_j)/\|\tilde{R}^3_i\times\tilde{R}^3_j\|$ and $\tilde{c}_{ji} =
\tilde{R}_j^{-1}\cdot (\tilde{R}^3_j\times
\tilde{R}^3_i)/\|\tilde{R}^3_j\times\tilde{R}^3_j\|$ from each pair
$\tilde{R}_i$ and $\tilde{R}_j$. After converting $\tilde{c}_{ij}$ and
$\tilde{c}_{ji}$   into $c_{ij}$ and $c_{ji}$, we replace $c_{ij}$ and $c_{ji}$
by two random vectors that are sampled from the uniform distribution over the
unit circle with probability $p$. That is, the common line vectors stay the same
with probability $(1-p)$.  We test the cases of $N=100, 500, 1000$. The computed objective function values are presented
in the left column of \cref{fig:cryoemfig}. The lines ``eigs'', ``IDDM mean'' and
``IDDM min'' are the objective function value computed by eigs, the averaged and minimum objective function value computed by
IDDM, respectively. The lines ``RSlocal mean'' and ``RSlocal min'' are the 
corresponding values of RSlocal. We can see that both RSlocal and IDDM can find
 better objective function values than eigs.
 We should point that both RSlocal and IDDM can find the same minumum when they start from
 the initial point generated by eigs. However, IDDM
 performs better than RSlocal on average.  
 
 A detailed summary of the computational results are reported in 
 \cref{tab:q05mserand}.  We further denote ``mse1'' as the smallest MSE
 generated in the ten cycles and ``obj1'' is the corresponding objective
 function value. Similarly,  ``obj2'' stands the smallest objective value in ten
 cycles and the corresponding MSE is denoted as ``mse2''. We can see that the
 pairs ``(mse1, obj1)'' and ``(mse2, obj2)'' are almost the same except the last
 row of each of $N=100,500,1000$. The reason is that the initial point produced by
 eigs lies in a small neighbourhood of the global solution and our local
 algorithm starting from eigs usually can find this global solution
 successfully. Although other cycles can also identify a local solution, the
 corresponding objective funtion values are larger. For the cases that mse1 is
 different from mse2, it means that a smaller objective function value does not
 necessary have a smaller MSE in the noisy cases. The reason is that the model
 \eqref{eq:cryoem} does not characterize the original Cryo-EM problem well.

Our second experiment is based on the dataset from \cite{Singer:2011ba}. The
noise-to-signal ratio (NSR) is defined as
$\textrm{NSR}=\textrm{Var}(Noise)/\textrm{Var}(Signal)$, where \textit{Signal}
is the clean projection image and \textit{Noise} is the noise realization. The
set up of the experiments is the same as the random data sets. The objective
function values are plotted in the right column of \cref{fig:cryoemfig}. They
show that eigs itself can provide a good solution when NSR is small. The
averaged objective function values obtained from IDDM are the best when NSR is
larger. IDDM also can find a smaller objective function value in a few cases.
 The detailed summary of computational results are presented in 
 \cref{tab:q05msesinger}. The
 pairs ``(mse1, obj1)'' and ``(mse2, obj2)'' are almost the same when NSR is
 small. However, for a large NSR, IDDM often is able to identify a smaller objective function
 value whose corresponding MSE is not the best. This observation again is not a
 contradiction but due to that the model
 \eqref{eq:cryoem} is not suitable in these cases.
 Nevertheless, these experiments are still perfect to show that IDDM is often better than the
 local algorithm itself and the local algorithm  starting from multiple
 randomly generated initial points when the global solution is difficult to
 be captured.

\begin{figure}[h]
\centering
\begin{minipage}{0.50\linewidth}
\includegraphics[width=1\linewidth]{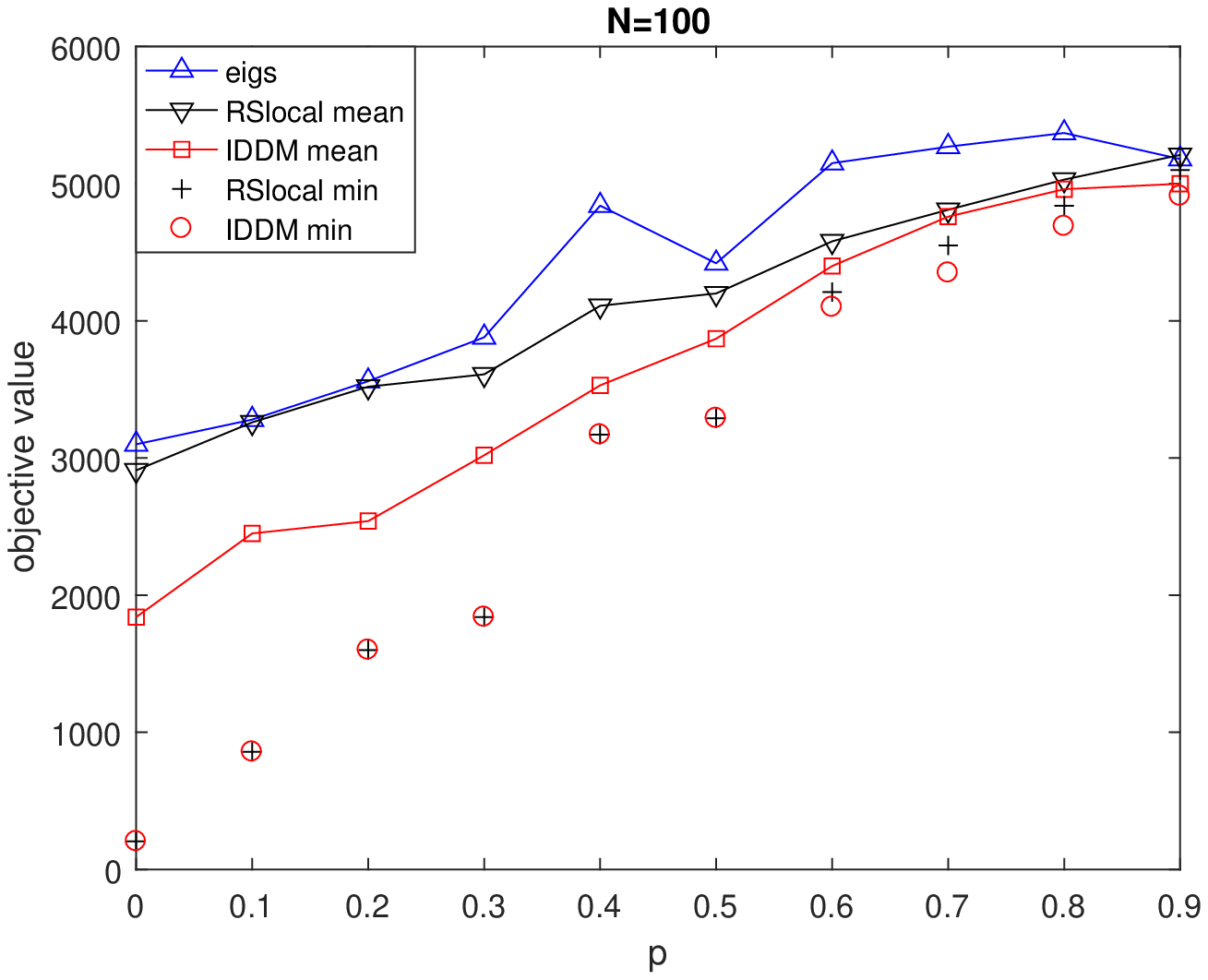}
\end{minipage}\hfill
\begin{minipage}{0.50\linewidth}
\includegraphics[width=1\linewidth]{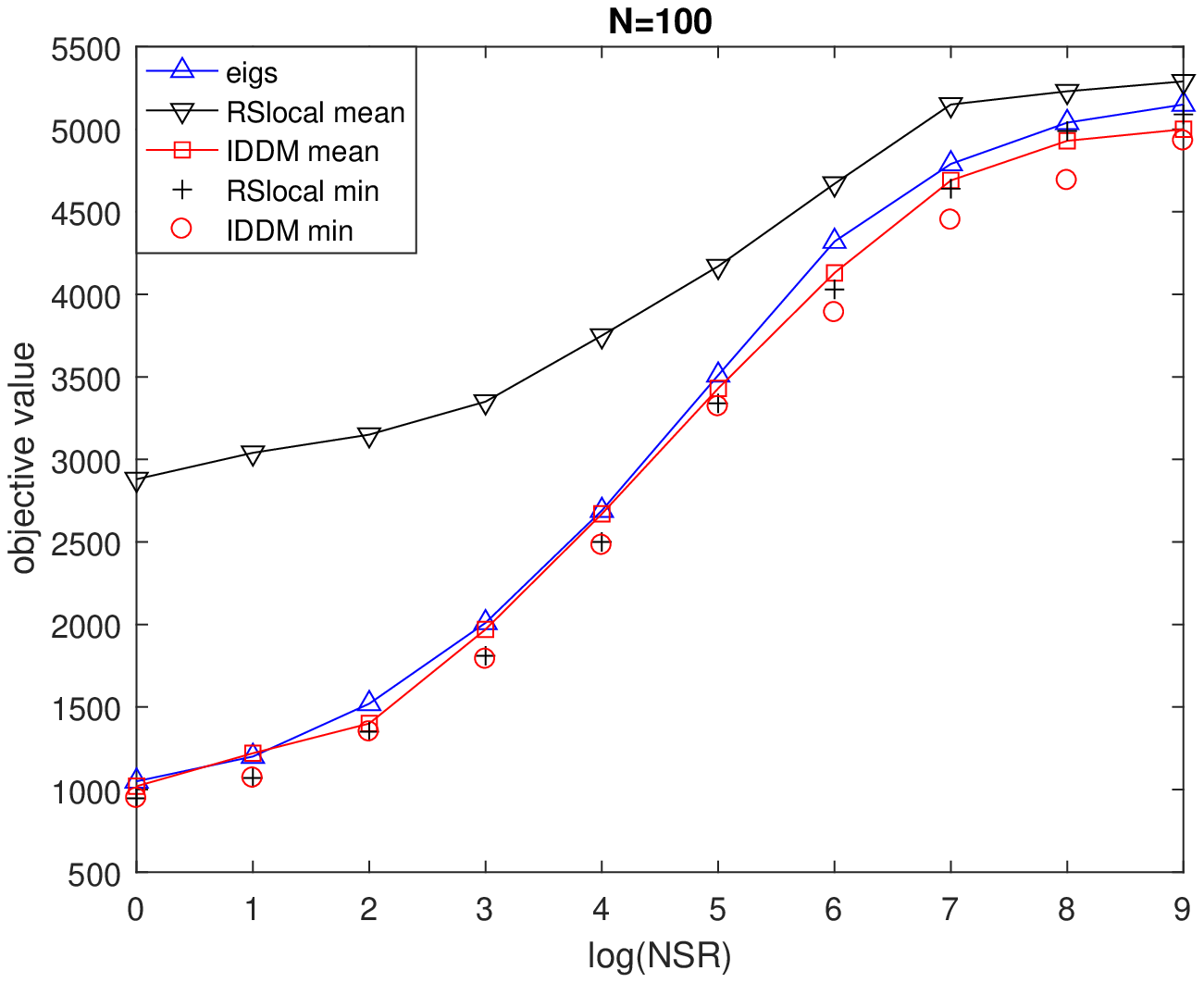}
\end{minipage}\hfill\\
\begin{minipage}{0.50\linewidth}
\includegraphics[width=1\linewidth]{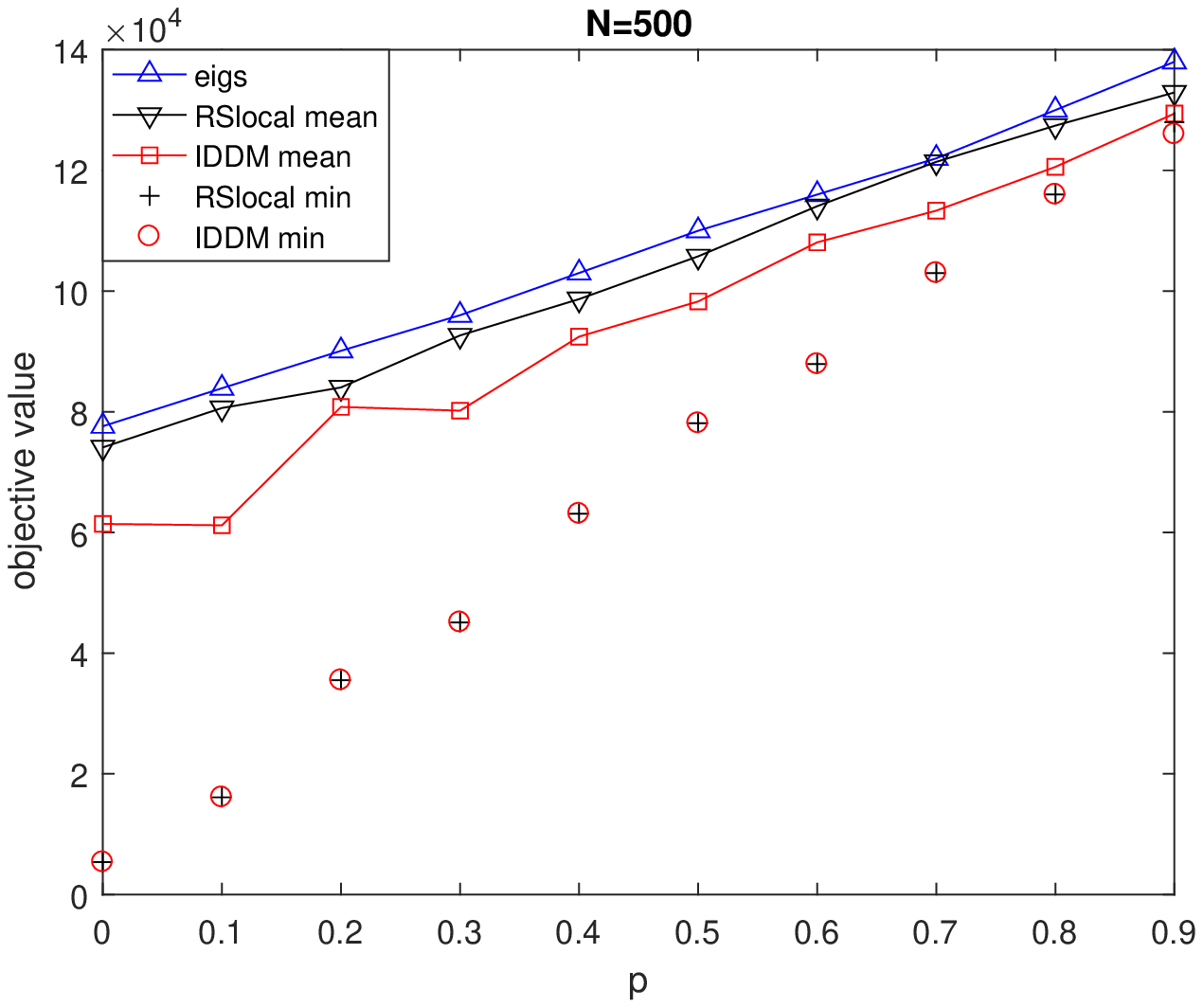}
\end{minipage}\hfill
\begin{minipage}{0.50\linewidth}
\includegraphics[width=1\linewidth]{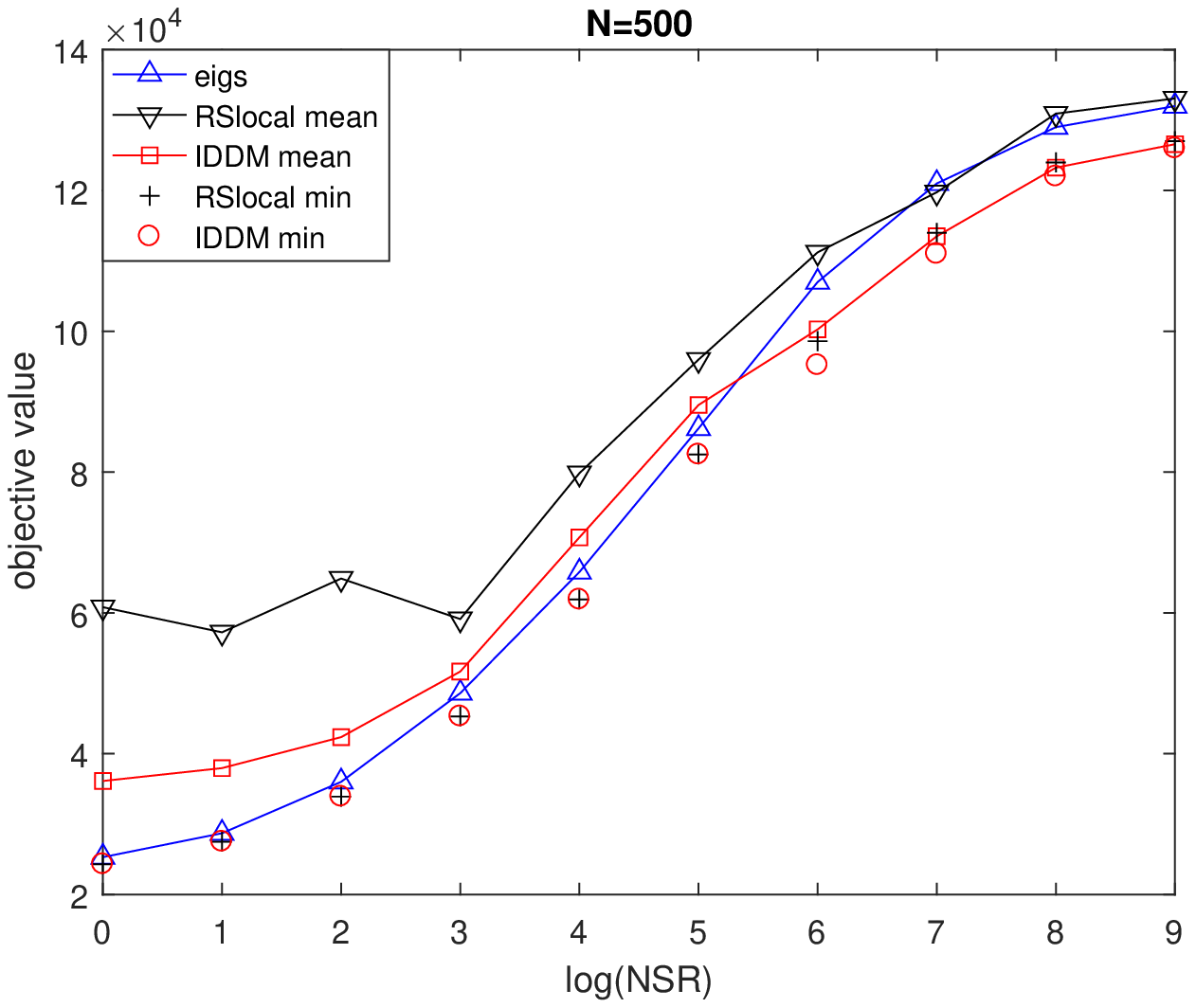}
\end{minipage}\hfill\\
\begin{minipage}{0.50\linewidth}
\includegraphics[width=1\linewidth]{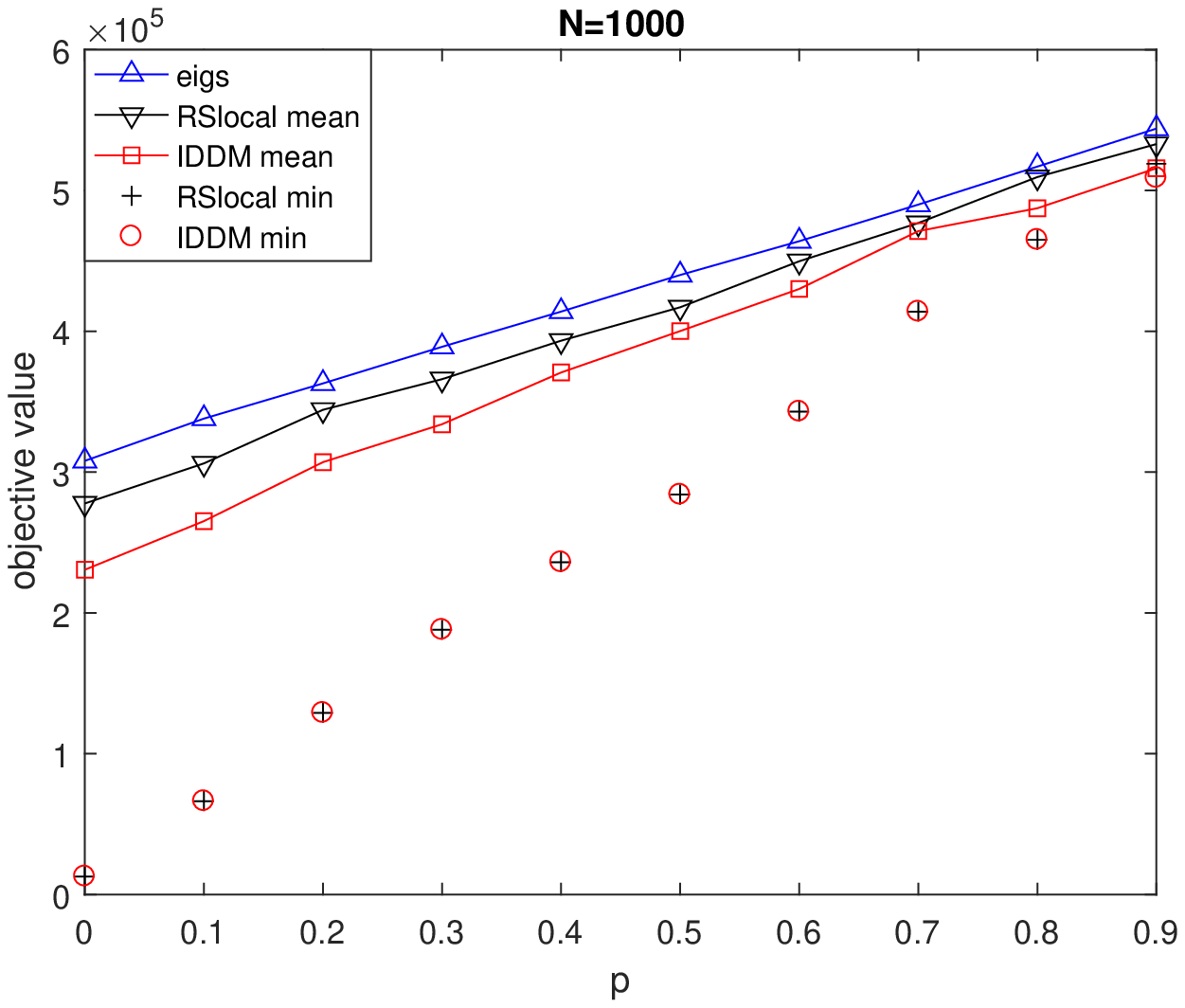}
\end{minipage}\hfill
\begin{minipage}{0.50\linewidth}
\includegraphics[width=1\linewidth]{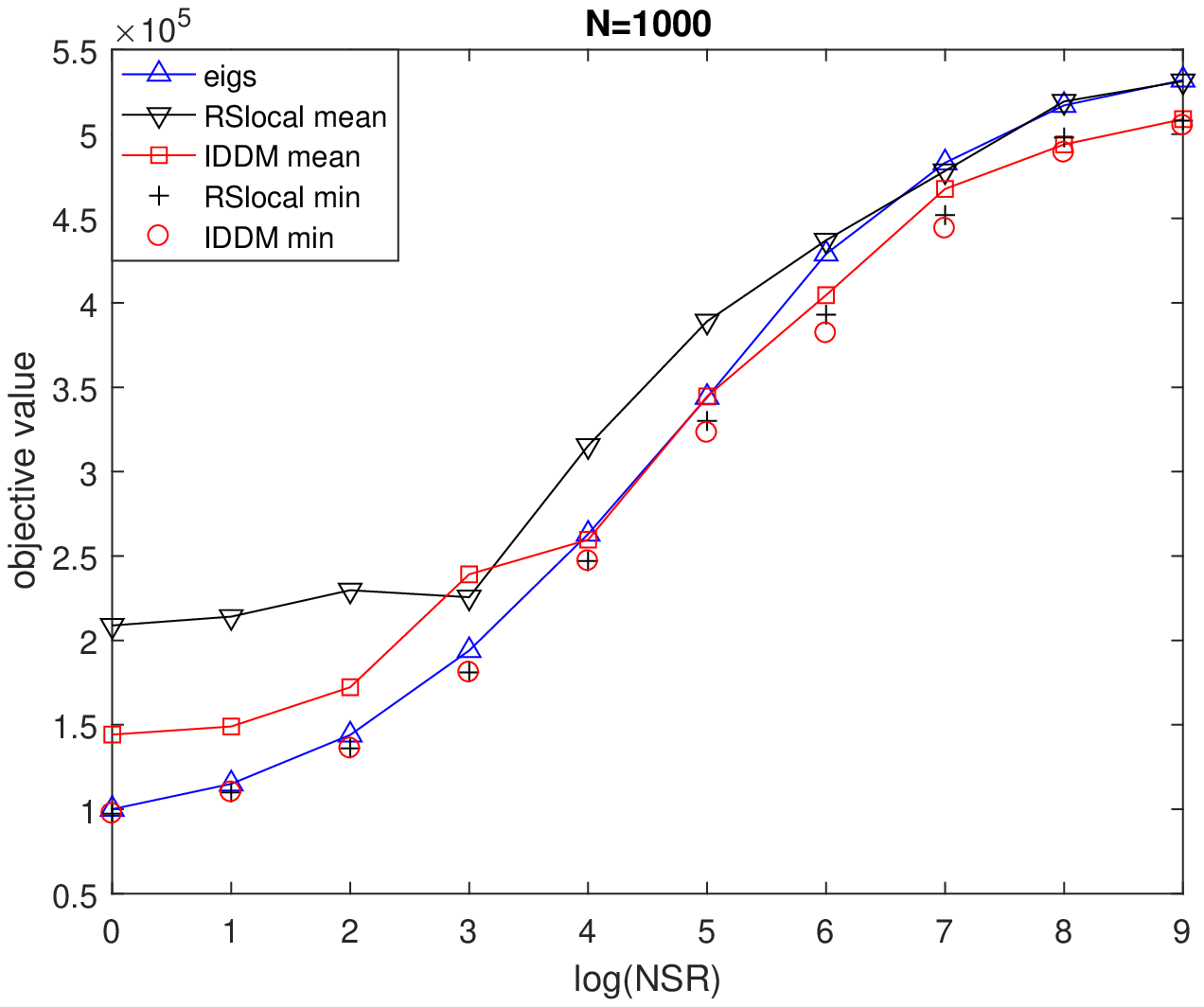}
\end{minipage}\hfill
\caption{The objective values for the random datasets (left column) and the
dataset from \cite{Singer:2011ba} (right column)}
\label{fig:cryoemfig}
\end{figure}

\begin{table}[htbp]
  \centering
  \footnotesize
\setlength{\tabcolsep}{3pt}
  \caption{The MSE of the eigenvector, RSlocal and IDDM for random dataset}
      \begin{tabular}{|c|c|c|cc|cc|c|cc|cc|c|}\hline
    \multirow{2}{*}{P}    & \multicolumn{2}{c|}{eigs}     & \multicolumn{5}{c|}{local} & \multicolumn{5}{c|}{IDDM} \\\cline{2-13}
    & mse & obj & \multicolumn{2}{c|}{mse1, obj1} & \multicolumn{2}{c|}{mse2, obj2} & cpu& \multicolumn{2}{c|}{mse1, obj1} & \multicolumn{2}{c|}{mse2, obj2} & cpu\\\hline

\multicolumn{13}{|c|}{N=100}\\ \hline
1.0 &  3.45e-1 & 3.10e3 & 1.19e-4 & 2.04e2 & 1.19e-4 & 2.04e2 &    0.7 & 1.19e-4 & 2.04e2 & 1.19e-4 & 2.04e2 &    0.8 \\ \hline 
0.9 &  3.31e-1 & 3.28e3 & 9.98e-4 & 8.57e2 & 9.98e-4 & 8.57e2 &    0.7 & 9.98e-4 & 8.57e2 & 9.98e-4 & 8.57e2 &    0.8 \\ \hline 
0.8 &  3.77e-1 & 3.56e3 & 7.01e-3 & 1.60e3 & 7.01e-3 & 1.60e3 &    0.7 & 7.01e-3 & 1.60e3 & 7.01e-3 & 1.60e3 &    0.7 \\ \hline 
0.7 &  3.97e-1 & 3.88e3 & 3.86e-3 & 1.84e3 & 3.86e-3 & 1.84e3 &    0.7 & 3.86e-3 & 1.84e3 & 3.86e-3 & 1.84e3 &    0.7 \\ \hline 
0.6 &  2.18 & 4.84e3 & 1.25 & 3.17e3 & 1.25 & 3.17e3 &    0.5 & 9.78e-1 & 3.26e3 & 1.25 & 3.17e3 &    0.6 \\ \hline 
0.5 &  7.00e-1 & 4.42e3 & 1.91e-1 & 3.29e3 & 1.91e-1 & 3.29e3 &    0.4 & 1.91e-1 & 3.29e3 & 1.91e-1 & 3.29e3 &    0.5 \\ \hline 
0.4 &  2.89 & 5.15e3 & 1.98 & 4.21e3 & 1.98 & 4.21e3 &    0.5 & 1.67 & 4.10e3 & 1.67 & 4.10e3 &    0.6 \\ \hline 
0.3 &  3.45 & 5.27e3 & 2.23 & 4.75e3 & 2.70 & 4.55e3 &    0.5 & 2.12 & 4.35e3 & 2.12 & 4.35e3 &    0.5 \\ \hline 
0.2 &  3.56 & 5.37e3 & 2.53 & 4.84e3 & 2.53 & 4.84e3 &    0.5 & 2.28 & 4.69e3 & 2.28 & 4.69e3 &    0.5 \\ \hline 
0.1 &  4.27 & 5.18e3 & 2.67 & 5.22e3 & 4.25 & 5.10e3 &    0.5 & 3.60 & 4.91e3 & 3.60 & 4.91e3 &    0.5 \\ \hline

\multicolumn{13}{|c|}{N=500}\\ \hline
1.0 &  3.42e-1 & 7.76e4 & 1.33e-4 & 5.35e3 & 1.33e-4 & 5.35e3 &    9.8 & 1.33e-4 & 5.35e3 & 1.33e-4 & 5.35e3 &    9.9 \\ \hline 
0.9 &  3.42e-1 & 8.39e4 & 4.98e-6 & 1.61e4 & 4.98e-6 & 1.61e4 &    9.9 & 4.98e-6 & 1.61e4 & 4.98e-6 & 1.61e4 &    8.5 \\ \hline 
0.8 &  3.40e-1 & 9.01e4 & 9.34e-4 & 3.55e4 & 9.34e-4 & 3.55e4 &    6.4 & 9.34e-4 & 3.55e4 & 9.34e-4 & 3.55e4 &    7.1 \\ \hline 
0.7 &  3.35e-1 & 9.60e4 & 3.42e-5 & 4.51e4 & 3.42e-5 & 4.51e4 &    6.3 & 3.42e-5 & 4.51e4 & 3.42e-5 & 4.51e4 &    7.4 \\ \hline 
0.6 &  3.74e-1 & 1.03e5 & 2.56e-3 & 6.31e4 & 2.56e-3 & 6.31e4 &    5.9 & 2.56e-3 & 6.31e4 & 2.56e-3 & 6.31e4 &    6.8 \\ \hline 
0.5 &  3.74e-1 & 1.10e5 & 6.28e-3 & 7.81e4 & 6.28e-3 & 7.81e4 &    5.6 & 6.28e-3 & 7.81e4 & 6.28e-3 & 7.81e4 &    5.9 \\ \hline 
0.4 &  3.89e-1 & 1.16e5 & 7.45e-3 & 8.79e4 & 7.45e-3 & 8.79e4 &    5.0 & 7.45e-3 & 8.79e4 & 7.45e-3 & 8.79e4 &    5.5 \\ \hline 
0.3 &  4.54e-1 & 1.22e5 & 2.22e-2 & 1.03e5 & 2.22e-2 & 1.03e5 &    4.3 & 2.22e-2 & 1.03e5 & 2.22e-2 & 1.03e5 &    5.2 \\ \hline 
0.2 &  8.04e-1 & 1.30e5 & 1.35e-1 & 1.16e5 & 1.35e-1 & 1.16e5 &    4.0 & 1.35e-1 & 1.16e5 & 1.35e-1 & 1.16e5 &    4.8 \\ \hline 
0.1 &  4.03 & 1.38e5 & 2.67 & 1.33e5 & 3.42 & 1.28e5 &    4.3 & 2.59 & 1.26e5 & 2.59 & 1.26e5 &    4.8 \\ \hline

\multicolumn{13}{|c|}{N=1000}\\ \hline
1.0 &  3.34e-1 & 3.08e5 & 1.43e-5 & 1.28e4 & 1.43e-5 & 1.28e4 &    163 & 1.43e-5 & 1.28e4 & 1.43e-5 & 1.28e4 &     87 \\ \hline 
0.9 &  3.51e-1 & 3.38e5 & 8.13e-6 & 6.63e4 & 8.13e-6 & 6.63e4 &    113 & 8.13e-6 & 6.63e4 & 8.13e-6 & 6.63e4 &     49 \\ \hline 
0.8 &  3.51e-1 & 3.63e5 & 9.46e-5 & 1.29e5 & 9.46e-5 & 1.29e5 &    103 & 9.46e-5 & 1.29e5 & 9.46e-5 & 1.29e5 &     57 \\ \hline 
0.7 &  3.54e-1 & 3.89e5 & 3.20e-4 & 1.88e5 & 3.20e-4 & 1.88e5 &     91 & 3.20e-4 & 1.88e5 & 3.20e-4 & 1.88e5 &     42 \\ \hline 
0.6 &  3.59e-1 & 4.14e5 & 1.51e-4 & 2.36e5 & 1.51e-4 & 2.36e5 &     84 & 1.51e-4 & 2.36e5 & 1.51e-4 & 2.36e5 &     33 \\ \hline 
0.5 &  3.67e-1 & 4.40e5 & 2.47e-5 & 2.84e5 & 2.47e-5 & 2.84e5 &     57 & 2.47e-5 & 2.84e5 & 2.47e-5 & 2.84e5 &     35 \\ \hline 
0.4 &  3.61e-1 & 4.64e5 & 1.05e-3 & 3.43e5 & 1.05e-3 & 3.43e5 &     50 & 1.05e-3 & 3.43e5 & 1.05e-3 & 3.43e5 &     27 \\ \hline 
0.3 &  3.97e-1 & 4.90e5 & 1.31e-2 & 4.14e5 & 1.31e-2 & 4.14e5 &     34 & 1.31e-2 & 4.14e5 & 1.31e-2 & 4.14e5 &     23 \\ \hline 
0.2 &  5.26e-1 & 5.17e5 & 5.00e-2 & 4.65e5 & 5.00e-2 & 4.65e5 &     18 & 5.00e-2 & 4.65e5 & 5.00e-2 & 4.65e5 &     18 \\ \hline 
0.1 &  2.24 & 5.44e5 & 1.44 & 5.25e5 & 2.61 & 5.19e5 &     19 & 1.44 & 5.25e5 & 3.91 & 5.09e5 &     18 \\ \hline \end{tabular}
	\label{tab:q05mserand}
\end{table}

\begin{table}[htbp]
  \centering
  \footnotesize
\setlength{\tabcolsep}{3pt}
  \caption{The MSE of the eigenvector, RSlocal and IDDM for dataset from \cite{Singer:2011ba}}
    \begin{tabular}{|c|c|c|cc|cc|c|cc|cc|c|}\hline
    \multirow{2}{*}{NSR}    & \multicolumn{2}{c|}{eigs}     & \multicolumn{5}{c|}{RSlocal} & \multicolumn{5}{c|}{IDDM} \\\cline{2-13}
    & mse & obj & \multicolumn{2}{c|}{mse1, obj1} & \multicolumn{2}{c|}{mse2, obj2} & cpu & \multicolumn{2}{c|}{mse1, obj1} & \multicolumn{2}{c|}{mse2, obj2} & cpu \\\hline
    
\multicolumn{13}{|c|}{N=100}\\ \hline
  1 &  3.00e-3 & 1.05e3 & 3.04e-4 & 9.47e2 & 3.04e-4 & 9.47e2 &    0.5 & 2.99e-4 & 9.47e2 & 3.04e-4 & 9.47e2 &    0.5 \\ \hline 
  2 &  4.41e-3 & 1.20e3 & 5.04e-4 & 1.07e3 & 5.04e-4 & 1.07e3 &    0.5 & 4.95e-4 & 1.07e3 & 5.15e-4 & 1.07e3 &    0.5 \\ \hline 
  4 &  1.06e-2 & 1.52e3 & 1.60e-3 & 1.35e3 & 1.60e-3 & 1.35e3 &    0.5 & 1.07e-3 & 1.35e3 & 1.10e-3 & 1.35e3 &    0.5 \\ \hline 
  8 &  2.88e-2 & 2.01e3 & 6.19e-3 & 1.81e3 & 6.19e-3 & 1.81e3 &    0.6 & 3.38e-3 & 1.79e3 & 3.38e-3 & 1.79e3 &    0.5 \\ \hline 
 16 &  8.81e-2 & 2.69e3 & 3.77e-2 & 2.50e3 & 3.77e-2 & 2.50e3 &    0.6 & 3.04e-2 & 2.48e3 & 3.12e-2 & 2.48e3 &    0.5 \\ \hline 
 32 &  2.63e-1 & 3.51e3 & 1.88e-1 & 3.34e3 & 1.88e-1 & 3.34e3 &    0.5 & 1.88e-1 & 3.34e3 & 2.07 & 3.32e3 &    0.4 \\ \hline 
 64 &  2.60 & 4.32e3 & 2.30 & 4.03e3 & 2.30 & 4.03e3 &    0.5 & 2.30 & 4.03e3 & 2.40 & 3.89e3 &    0.4 \\ \hline 
128 &  3.28 & 4.79e3 & 3.08 & 4.64e3 & 3.08 & 4.64e3 &    0.4 & 3.08 & 4.64e3 & 3.40 & 4.45e3 &    0.4 \\ \hline 
256 &  4.10 & 5.04e3 & 4.05 & 4.99e3 & 4.05 & 4.99e3 &    0.4 & 4.05 & 4.99e3 & 4.19 & 4.69e3 &    0.4 \\ \hline 
512 &  4.97 & 5.15e3 & 4.98 & 5.09e3 & 4.98 & 5.09e3 &    0.4 & 4.86 & 5.03e3 & 5.16 & 4.93e3 &    0.4 \\ \hline

\multicolumn{13}{|c|}{N=500}\\ \hline
  1 &  1.20e-3 & 2.53e4 & 1.54e-4 & 2.43e4 & 1.54e-4 & 2.43e4 &     16 & 1.54e-4 & 2.43e4 & 1.54e-4 & 2.43e4 &    7.8 \\ \hline 
  2 &  1.77e-3 & 2.87e4 & 2.94e-4 & 2.75e4 & 2.94e-4 & 2.75e4 &     17 & 2.94e-4 & 2.75e4 & 2.94e-4 & 2.75e4 &     11 \\ \hline 
  4 &  5.10e-3 & 3.60e4 & 7.58e-4 & 3.39e4 & 7.58e-4 & 3.39e4 &     16 & 7.07e-4 & 3.39e4 & 7.07e-4 & 3.39e4 &    7.1 \\ \hline 
  8 &  1.91e-2 & 4.86e4 & 3.14e-3 & 4.53e4 & 3.14e-3 & 4.53e4 &     17 & 3.14e-3 & 4.53e4 & 3.14e-3 & 4.53e4 &    7.1 \\ \hline 
 16 &  6.35e-2 & 6.58e4 & 1.81e-2 & 6.19e4 & 1.81e-2 & 6.19e4 &    8.7 & 1.81e-2 & 6.19e4 & 1.81e-2 & 6.19e4 &    7.6 \\ \hline 
 32 &  2.18e-1 & 8.62e4 & 1.34e-1 & 8.25e4 & 1.34e-1 & 8.25e4 &    5.5 & 1.34e-1 & 8.25e4 & 1.34e-1 & 8.25e4 &    6.8 \\ \hline 
 64 &  1.75 & 1.07e5 & 1.61 & 9.86e4 & 1.61 & 9.86e4 &    5.6 & 1.61 & 9.86e4 & 2.23 & 9.52e4 &    9.3 \\ \hline 
128 &  2.62 & 1.21e5 & 2.32 & 1.14e5 & 2.32 & 1.14e5 &    5.7 & 2.32 & 1.14e5 & 2.84 & 1.11e5 &    8.8 \\ \hline 
256 &  3.49 & 1.29e5 & 3.22 & 1.24e5 & 3.22 & 1.24e5 &    5.0 & 3.22 & 1.24e5 & 4.18 & 1.22e5 &    5.5 \\ \hline 
512 &  4.59 & 1.32e5 & 4.65 & 1.32e5 & 4.84 & 1.27e5 &    6.1 & 4.58 & 1.26e5 & 4.58 & 1.26e5 &    5.9 \\ \hline

\multicolumn{13}{|c|}{N=1000}\\ \hline
  1 &  8.27e-4 & 1.00e5 & 1.25e-4 & 9.73e4 & 1.25e-4 & 9.73e4 &     59 & 1.25e-4 & 9.73e4 & 1.25e-4 & 9.73e4 &     51 \\ \hline 
  2 &  1.46e-3 & 1.15e5 & 2.58e-4 & 1.10e5 & 2.58e-4 & 1.10e5 &     75 & 2.58e-4 & 1.10e5 & 2.58e-4 & 1.10e5 &     67 \\ \hline 
  4 &  4.56e-3 & 1.44e5 & 6.60e-4 & 1.36e5 & 6.60e-4 & 1.36e5 &     65 & 6.60e-4 & 1.36e5 & 6.60e-4 & 1.36e5 &     58 \\ \hline 
  8 &  1.81e-2 & 1.94e5 & 2.46e-3 & 1.81e5 & 2.46e-3 & 1.81e5 &     60 & 2.46e-3 & 1.81e5 & 2.46e-3 & 1.81e5 &     45 \\ \hline 
 16 &  6.41e-2 & 2.63e5 & 1.43e-2 & 2.47e5 & 1.43e-2 & 2.47e5 &     43 & 1.43e-2 & 2.47e5 & 1.43e-2 & 2.47e5 &     55 \\ \hline 
 32 &  2.32e-1 & 3.44e5 & 1.43e-1 & 3.30e5 & 1.43e-1 & 3.30e5 &     28 & 1.43e-1 & 3.30e5 & 2.09 & 3.23e5 &     35 \\ \hline 
 64 &  1.78 & 4.29e5 & 1.64 & 3.93e5 & 1.64 & 3.93e5 &     32 & 1.64 & 3.93e5 & 2.24 & 3.82e5 &     37 \\ \hline 
128 &  2.50 & 4.83e5 & 2.24 & 4.52e5 & 2.24 & 4.52e5 &     36 & 2.24 & 4.52e5 & 2.67 & 4.44e5 &     30 \\ \hline 
256 &  3.48 & 5.17e5 & 3.21 & 4.98e5 & 3.21 & 4.98e5 &     33 & 3.21 & 4.98e5 & 4.10 & 4.89e5 &     37 \\ \hline 
512 &  4.62 & 5.32e5 & 4.62 & 5.32e5 & 4.79 & 5.08e5 &     30 & 4.62 & 5.32e5 & 4.86 & 5.05e5 &     29 \\ \hline 
	\end{tabular}
    \label{tab:q05msesinger}
\end{table}

\section{Conclusion}
\label{sec:conclusion}
The goal of this paper is to construct an algorithm which is able to identify
global solutions of minimization with orthogonality constraints.   Our strategy
is simply alternating between a local algorithm on Stiefel manifold and a
gradient flow method with stochastic diffusion on manifold.  
 The main concept is that a suitable diffusion term is able to drive the
 iteration to escape the region around a local solution. We derive an extrinsic
 form of the Brownian motion on the manifold and design a numerical efficient
 scheme to solve the corresponding SDE on manifold. We further theoretically 
 show the half order convergence of the proposed numerical method for solving 
 SDE on the Stefiel manifold. Moreover, convergence to the
 global minimizer is also theoretically established as long as the diffusion is sufficiently
 enough. However, our extensive numerical experiments on polynomial optimization
 and 3D structure determination from Cryo-EM show that a few cycles of our
 algorithm is often able to provide a better solution than the local algorithm.
 Although both theoretical and numerical results are still limited in certain
 senses, they are indeed promising especially for problems with good structures.
 Our future work includes a better theoretical understanding the algorithms,
 refining them for more typical applications and some better ways on choosing
  or even learning the diffusion parameter $\sigma(t)$.

\clearpage
\bibliographystyle{siamplain}
\bibliography{GlobalOptStfMfd.bib}
\end{document}